\documentclass[10pt,a4paper]{amsart}

\usepackage{amstext,amssymb,amsopn,amsthm,mathrsfs,amsfonts,amsmath,amsxtra}

\usepackage{dsfont,multirow,comment,enumerate,bm,float}

\usepackage{graphicx,color}

\allowdisplaybreaks

\newtheorem{theorem}{Theorem}[section]
\newtheorem{corollary}[theorem]{Corollary}
\newtheorem{proposition}[theorem]{Proposition}

\newtheorem{lemma}[theorem]{Lemma}
\newtheorem{remark}[theorem]{Remark}

\DeclareMathOperator{\spann}{span}
\DeclareMathOperator{\domain}{Dom}

\DeclareMathOperator{\id}{Id}

\def\a{\alpha}
\def\b{\beta}
\def\z{\lambda}

\title[Fourier-Dini heat kernel estimates]
	{On sharp heat kernel estimates \\
	in the context of Fourier-Dini expansions}

\author[B.\ Langowski]{Bartosz Langowski}

\address{Bartosz Langowski, \newline
Department of Mathematics and Physical Sciences, \newline
Franciscan University of Steubenville, \newline
1235 University Blvd., Steubenville, OH 43952, USA
}
\email{blangowski@franciscan.edu}

\author[A.\ Nowak]{Adam Nowak}

\address{Adam Nowak, \newline
			Institute of Mathematics, \newline
      Polish Academy of Sciences, \newline
      \'Sniadeckich 8,
      00--656 Warszawa, Poland 
}
\email{anowak@impan.pl}

\begin{document}


\begin{abstract}
We prove sharp estimates of the heat kernel associated with Fourier-Dini expansions on $(0,1)$ equipped with Lebesgue
measure and the Neumann condition imposed on the right endpoint. Then we give several applications of this result including
sharp bounds for the corresponding Poisson and potential kernels, sharp mapping properties of the maximal heat semigroup
and potential operators and boundary convergence of the Fourier-Dini semigroup.
\end{abstract}

\maketitle
\thispagestyle{empty}

\footnotetext{
\emph{2020 Mathematics Subject Classification:} primary 42C05; secondary 35K08.\\
\emph{Key words and phrases:} 
Fourier-Dini system, heat semigroup, heat kernel, sharp estimate,
maximal operator, Poisson kernel, potential kernel, potential operator.
}

\section{Introduction} \label{sec:intro}

Discrete Fourier-Bessel and Fourier-Dini expansions on the interval $(0,1)$ have been present in the literature since
the 19th century, cf.\ \cite[Chap.\,XVIII]{watson}, and they are of great importance in analysis and applications.
Roughly, both types of expansions form in fact a single family in which Fourier-Bessel corresponds to the Dirichlet
boundary condition at the right endpoint of $(0,1)$, while Fourier-Dini to all the remaining Robin boundary conditions
including the Neumann condition.

Analysis related to Fourier-Bessel expansions was a subject of numerous papers, see e.g.\ our recent work \cite{LN}
and references given there. In particular, estimates for the corresponding heat kernel were studied in \cite{NR1,NR2},
and finally sharp bounds were obtained in \cite{MSZ}. The result of \cite{MSZ} also follows from \cite{NR2} and the recent
sharp estimates of the Jacobi heat kernel obtained gradually in \cite{NSS1,NSS2,NSS3}.
On the other hand, Fourier-Dini expansions have more complicated structure and their analysis is far less developed.
The present paper appears to be the first that treats the issues indicated in the abstract in the
framework of Fourier-Dini expansions.

We work in the setting of Fourier-Dini expansions on $(0,1)$ equipped with Lebesgue measure.
Our main results pertain to the situation where the Neumann
condition is imposed on the right endpoint in the sense that all functions of the Fourier-Dini
system satisfy this condition. The central outcome of the paper, Theorem \ref{thm:main}, provides sharp global estimates of
the Fourier-Dini heat kernel viz.\ the integral kernel of the Fourier-Dini semigroup. This kernel is
given only implicitly, by a heavily oscillating series involving Bessel functions and zeros of some related functions,
and no direct insight into its behavior seems to be possible. To prove the short-time bounds, which are the essence of the whole
result, we use an indirect method having its roots in \cite{NR2}. This method relies on establishing a connection
with a situation of expansions based on Jacobi polynomials and then transferring known sharp bounds for the
related Jacobi heat kernel.
As an auxiliary result, we also show that \cite{NR2} together with the present paper completely exhaust applicability
of the method within the general framework of Fourier-Bessel/Fourier-Dini expansions.

Our principal motivation for investigating the Fourier-Dini heat kernel comes from an interest in harmonic
analysis related to Fourier-Dini expansions. The sharp result we prove has far reaching consequences and we point out some
of them as applications. These include sharp estimates for the related Poisson kernels (Theorem \ref{thm:poisson})
and potential kernels (Theorem \ref{thm:potk}), and the latter lead to a complete characterization of $L^p-L^q$
mapping properties of Fourier-Dini potential operators (Theorem \ref{thm:pot}).
Further consequences pertain to (optimal) $L^p$ mapping properties of the Fourier-Dini semigroup maximal operators
(Theorem \ref{thm:max}) and hence boundary a.e.\ convergence of the semigroup (Corollary \ref{cor:conv}).

The paper is organized as follows. In Section \ref{sec:pre} we introduce basic notation and gather various
definitions, facts and results related to Bessel functions and Fourier-Dini systems.
In Section \ref{sec:main} we state the main result, Theorem \ref{thm:main}, and give its proof.
We also discuss the issue of further inapplicability of the proof method within the framework of Fourier-Dini expansions.
This actually requires proving several auxiliary results concerning domains of certain unbounded operators,
which are postponed to the last Section \ref{sec:dom}.
Finally, in Section \ref{sec:app} we present in detail all the announced applications of the main result.


\section{Preliminaries} \label{sec:pre}

In this section we first briefly describe the general notation used in the paper.
Then we present for further reference various facts and formulas concerning Bessel functions.
Finally, we give basic information on Fourier-Dini systems and some related objects.

\subsection{Notation}

Throughout this paper we use fairly standard notation.
The minimum of two quantities will be indicated by $\wedge$.
By $\langle \cdot, \cdot \rangle$ we denote the usual inner product in $L^2((0,1),dx)$.
All $L^p$ spaces in this paper are over the interval $(0,1)$ equipped with Lebesgue measure. 

For a $\textsf{condition}$ that can have logical value true or false we shall denote
$$
\chi_{\{\textsf{condition}\}} :=
	\begin{cases}
	1, & \textrm{if} \;\, \textsf{condition} \;\, \textrm{is true}, \\
	0, & \textrm{if} \;\, \textsf{condition} \;\, \textrm{is false}.
	\end{cases}
$$

We write $X \lesssim Y$ to indicate that $X \le C Y$ with a positive constant $C$ independent of significant quantities.
We shall write $X \simeq Y$ when simultaneously $X \lesssim Y$ and $Y \lesssim X$.

\subsection{Facts and formulas concerning the Bessel functions $J_{\nu}$ and $I_{\nu}$} \label{sec:Bes}
For the material presented in this section we refer to Watson's monograph \cite{watson}; for an easy
online access to many facts and formulas related to Bessel functions, see \cite[Chap.\,10]{handbook}.

Let $J_{\nu}$ be the Bessel function of the first kind and order $\nu$ and let
$I_{\nu}$ be the modified Bessel function of the first kind and order $\nu$. In this paper, for our purpose,
we shall essentially always consider $J_{\nu}$ and $I_{\nu}$ as functions on the positive half-line.
Further, throughout the paper we assume that the parameter $\nu > -1$.
These restrictions could sometimes be weakened or released, but we will not go into details.

For further reference, we list some useful identities:
\begin{align}
x J'_{\nu}(x) - \nu J_{\nu}(x) & = -x J_{\nu+1}(x), & x I'_{\nu}(x) - \nu I_{\nu}(x) & = x I_{\nu+1}(x),
\label{id3}\\
\big[ x^{-\nu} J_{\nu}(x) \big]' & = -x^{-\nu} J_{\nu+1}(x), &  \big[ x^{-\nu} I_{\nu}(x)\big]' & = x^{-\nu} I_{\nu+1}(x). \label{id5}
\end{align}
We also note the asymptotics
\begin{align} \label{as1}
J_{\nu}(x) & = \frac{x^{\nu}}{2^{\nu}\Gamma(\nu+1)} + \mathcal{O}\big(x^{\nu+2}\big), &
I_{\nu}(x) & = \frac{x^{\nu}}{2^{\nu}\Gamma(\nu+1)} + \mathcal{O}\big(x^{\nu+2}\big), & x \to 0^+, \\
\label{as2}
J_{\nu}(x) & = \mathcal{O}\big(x^{-1/2}\big), &
I_{\nu}(x) & = \frac{e^x}{\sqrt{2\pi x}}\Big( 1 + \mathcal{O}\big( 1/x \big) \Big), & x \to \infty,
\end{align}
and the special cases
\begin{equation} \label{expl}
J_{-1/2}(x) = \sqrt{\frac{2}{\pi x}} \cos x, \qquad J_{1/2}(x) = \sqrt{\frac{2}{\pi x}} \sin x.
\end{equation}
In general, $J_{\nu}$ expresses directly via elementary functions if and only if $\nu$ is an odd half-integer.
Both $J_{\nu}$ and $I_{\nu}$ are smooth on $(0,\infty)$ and $I_{\nu}$ is strictly positive there.

For a parameter $H \in \mathbb{R}$ we consider the auxiliary functions
$$
J_{\nu,H}(x) := x J'_{\nu}(x) + H J_{\nu}(x) \qquad \textrm{and} \qquad
I_{\nu,H}(x) := x I'_{\nu}(x) + H I_{\nu}(x).
$$
Making use of \eqref{id3} we get
\begin{align} \label{id6}
J_{\nu,H}(x) & = x^{1-H} \big[ x^{H} J_{\nu}(x)\big]' = (H+\nu) J_{\nu}(x) - x J_{\nu+1}(x), \\
I_{\nu,H}(x) & = x^{1-H} \big[ x^{H} I_{\nu}(x)\big]' = (H+\nu) I_{\nu}(x) + x I_{\nu+1}(x).
\label{id66}
\end{align}

The function $J_{\nu,H}$ has infinitely many isolated zeros in $(0,\infty)$. We denote their sequence in an ascending order
of magnitude by $\{\z_n^{\nu,H} : n \ge 1\}$.
The function $I_{\nu,H}$ has no zeros in $(0,\infty)$ when $H + \nu \ge 0$, and exactly one
strictly positive zero otherwise$^{\dag}$ which we denote by $\z_0^{\nu,H}$
(note that$^{\ddag}$ when $\nu+H < 0$, $\pm i \z_0^{\nu,H}$ are the only imaginary zeros of $(\cdot)^{-\nu} J_{\nu,H}$,
cf.\ \cite[Chap.\,XVIII, Sec.\,18{$\cdot$3}]{watson}).
In case $H+\nu = 0$, we also set $\z_0^{\nu,H} := 0$.
\footnote{$\dag$ This can immediately be seen from \eqref{id66} and the Mittag-Leffler expansion
$\frac{x I_{\nu+1}(x)}{I_{\nu}(x)} = \sum_{k=1}^{\infty} \frac{2}{1+(\z_{k}^{\nu}/ x)^2}$.}
\footnote{$\ddag$ It is easy to check that $(ix)^{-\nu}J_{\nu,H}(ix) = x^{-\nu}I_{\nu,H}(x)$ for $x>0$
and $H \in \mathbb{R}$.}

From Dixon's theorem, cf.\ \cite[Chap.\,XV, Sec.\,15$\cdot$23]{watson}, it follows that for each fixed $H \in \mathbb{R}$
the zeros of $J_{\nu}$ interlace those of $J_{\nu,H}$. Thus, by the well known asymptotics for the zeros of $J_{\nu}$
(see \cite[Chap.\,XV, Sec.\,15$\cdot$53]{watson}) we conclude that
\begin{equation} \label{azer}
\z_n^{\nu,H} = \pi n + \mathcal{O}(1), \qquad n \ge 1.
\end{equation}

There is a vast literature devoted to zeros of Bessel and related functions,
which covers also functions like $J_{\nu,H}$. The basic reference is Watson's monograph \cite{watson}.
For more recent developments concerning zeros of $J_{\nu,H}$,
see for instance \cite{landau} and references given there.

\subsection{Fourier-Dini system} \label{ssec:FD}

Define the quantities
$$
c_n^{\nu,H} := \frac{\sqrt{2}}{|J_{\nu}(\z_n^{\nu,H})|} \frac{\z_n^{\nu,H}}{\sqrt{\big(\z_n^{\nu,H}\big)^2-\nu^2+H^2}}, \qquad n \ge 1,
$$
and, in case $\nu + H < 0$, define also
$$
c_0^{\nu,H} := \frac{\sqrt{2}}{I_{\nu}(\z_0^{\nu,H})} \frac{\z_0^{\nu,H}}{\sqrt{\big(\z_0^{\nu,H}\big)^2 + \nu^2-H^2}}.
$$
Note that the constants $c_n^{\nu,H}$, $H \in \mathbb{R}$, $n \ge \chi_{\{H+\nu \ge 0\}}$,
are well-defined and strictly positive. This follows from evaluation of $(c_n^{\nu,H})^{-2}$
as certain strictly positive integrals, see the comments below.

Next, define the following functions on the interval $(0,1)$:
$$
\psi_n^{\nu,H}(x) := c_n^{\nu,H} \sqrt{x} J_{\nu}\big(\z_n^{\nu,H}x\big), \qquad n \ge 1.
$$
In addition, in case $\nu + H \le 0$, define
$$
\psi_0^{\nu,H}(x) :=
	\begin{cases}
		c_0^{\nu,H} \sqrt{x} I_{\nu}\big(\z_0^{\nu,H}x\big), & \;\; \textrm{if} \;\; H+\nu < 0, \\
		\sqrt{2(\nu+1)}\, x^{\nu+1/2}, & \;\; \textrm{if} \;\; H+\nu = 0.
	\end{cases}
$$

The result stated below is well-known, see \cite[Chap.\,XVIII]{watson} and \cite{H}.
The difficult part of the proof is completeness of the Fourier-Dini system, cf.\ \cite{H}.
On the other hand, orthogonality and values of the normalizing constants $c_n^{\nu,H}$ are deduced from Lommel's integrals, cf.\
\cite[Chap.\,V, Sec.\,5$\cdot$11]{watson}. Convenient references to the formulas needed here are the following:
\cite[Sec.\,1.8.3, Form.\,10]{prud2}, \cite[Sec.\,1.11.5, Form.\,2]{prud2},
\cite[Sec.\,1.8.1, Form.\,21]{prud2} (orthogonality) and
\cite[Sec.\,1.8.3, Form.\,11]{prud2}, \cite[Sec.\,1.11.3, Form.\,4]{prud2} (values of the normalizing constants).
\begin{theorem} \label{thm:onb}
Let $\nu > -1$ and $H \in \mathbb{R}$ be fixed. Then the Fourier-Dini system
$$
\big\{ \psi_n^{\nu,H} : n \ge \chi_{\{H+\nu > 0\}} \big\}
$$
is an orthonormal basis in $L^2((0,1),dx)$.
\end{theorem}

The next result is also known and can be easily verified by a direct computation with the aid of \eqref{id5}.
It says that the Fourier-Dini system consists of eigenfunctions of the Bessel differential operator
\begin{equation} \label{id8}
\mathbb{L}^{\nu}f(x) := -\frac{d^2}{dx^2}f(x) - \frac{1/4-\nu^2}{x^2}f(x)
	= - x^{-\nu-1/2} \frac{d}{dx} \Big( x^{2\nu+1} \frac{d}{dx} \big[ x^{-\nu-1/2}f(x)\big] \Big).
\end{equation}
\begin{proposition} \label{prop:eigen}
Let $\nu > -1$ and $H \in \mathbb{R}$. Then
$$
\mathbb{L}^{\nu} \psi_n^{\nu,H}
= (-1)^{\chi_{\{n=0\}}}\big(\z_n^{\nu,H}\big)^2 \psi_n^{\nu,H}, \qquad n \ge \chi_{\{H+\nu > 0\}}.
$$
\end{proposition}

From the divergence form \eqref{id8} of $\mathbb{L}^{\nu}$ it is easily seen that the Bessel operator is symmetric
and non-negative on $C_c^2(0,1) \subset L^2((0,1),dx)$. The operator $\mathbb{L}^{\nu}$ acting on the domain 
$C_c^2(0,1)$ has a self-adjoint extension $\widetilde{\mathbb{L}}^{\nu,H}$ naturally associated with 
the Fourier-Dini system $\{\psi_n^{\nu,H}\}$ and defined via its spectral decomposition as
$$
\widetilde{\mathbb{L}}^{\nu,H}f = \sum_{n \ge \chi_{\{H+\nu > 0\}}}
	(-1)^{\chi_{\{n=0\}}} \big( \z_n^{\nu,H}\big)^2 \big\langle f, \psi_n^{\nu,H}\big\rangle \psi_n^{\nu,H}
$$
on the domain$^{\sharp}$
$$
\domain \widetilde{\mathbb{L}}^{\nu,H} = \bigg\{ f \in L^2((0,1),dx) : \sum_{n \ge \chi_{\{H+\nu > 0\}}}
	\Big| \big( \z_n^{\nu,H}\big)^2 \big\langle f, \psi_n^{\nu,H}\big\rangle \Big|^2 < \infty \bigg\}.
$$
To see that this is indeed the extension from $C_c^2(0,1)$, observe that
$(\z_n^{\nu,H})^2 \langle f, \psi_n^{\nu,H}\rangle = \langle\mathbb{L}^{\nu}f,\psi_n^{\nu,H}\rangle$
for any $f \in C_c^2(0,1)$.
Note that $\widetilde{\mathbb{L}}^{\nu,H}$ is not non-negative in case $\nu + H < 0$, but it is otherwise.
Note also that despite of the lack of non-negativity when $\nu + H <0$, the operator $\widetilde{\mathbb{L}}^{\nu,H}$
is semi-bounded from below.
\footnote{$\sharp$
The domain $\domain \widetilde{\mathbb{L}}^{\nu,H}$ can be thought of as the maximal domain in $L^2$ sense for the expression
defining $\widetilde{\mathbb{L}}^{\nu,H}$, i.e.\ it consists of all functions $f \in L^2((0,1),dx)$ such that the formula
defining $\widetilde{\mathbb{L}}^{\nu,H}f$ yields a function in $L^2((0,1),dx)$
(with convergence of the series being understood in the $L^2$ sense).
}
Let $T_t^{\nu,H} = \exp(-t\widetilde{\mathbb{L}}^{\nu,H})$, $t \ge 0$,
be the semigroup generated by $- \widetilde{\mathbb{L}}^{\nu,H}$.
Then one has the integral representation 
$$
T_t^{\nu,H}f(x) = \int_0^1 G_t^{\nu,H}(x,y) f(y) \, dy, \qquad x \in (0,1), \quad t > 0,
$$
with the kernel given by the oscillatory series
\begin{equation} \label{serG}
G_t^{\nu,H}(x,y) = \sum_{n \ge \chi_{\{H+\nu > 0\}}} \exp\Big({-t (-1)^{\chi_{\{n=0\}}}\big(\z_n^{\nu,H}\big)^2}\Big)
	\psi_n^{\nu,H}(x) \psi_n^{\nu,H}(y), \qquad x,y \in (0,1), \quad t > 0.
\end{equation}
With the aid of \eqref{as2} and \eqref{azer} it can be seen that it is legitimate to differentiate the above series in $x,y$
and $t$ arbitrarily many times, hence $G_t^{\nu,H}(x,y)$ is a smooth function of $(x,y,t) \in (0,1)^2 \times (0,\infty)$.
We call $G_t^{\nu,H}(x,y)$ the heat kernel associated with Fourier-Dini expansions since $u(x,t) := T_t^{\nu,H}f(x)$
solves the heat equation based on $\mathbb{L}^{\nu}$, with initial values $u(\cdot,0)$ prescribed by $f$, the Robin boundary condition
$$\Big(H-\frac{1}2\Big)u(1,t) + u_x(1,t)=0$$ at the right endpoint,
and the condition
$$x^{\nu+1/2}\Big[\frac{\nu+1/2}x u(x,t) - u_x(x,t)\Big]\bigg|_{x=0^+}=0$$ at the left endpoint.
Note that for $\nu = - 1/2$ and for $\nu=1/2$ one recovers the Neumann and the Dirichlet$^\flat$ conditions at $x=0^+$,
respectively.
\footnote{$\flat$
Let $F(x)=xf(x)$. Under mild assumptions on $f$, the condition $F'(0^+)=0$ is equivalent to the Dirichlet condition $f(0^+)=0$.
Indeed, if $f(0^+)=0$, then $F(0^+)=0$ and $F'(0^+) = \lim_{x\to 0^+}\frac{F(x)-F(0^+)}{x}=0$.
On the other hand, if $F'(0^+)=0$, then e.g.\ local integrability of $f$ forces $F(0^+)=0$ and in this case
$f(0^+) = \lim_{x\to 0^+} \frac{F(x)-F(0^+)}{x} = F'(0^+)=0$.
}

\section{Main result} \label{sec:main}

We shall prove the following sharp bounds.
\begin{theorem} \label{thm:main}
Assume that $\nu > -1$. Given any $T > 0$, one has
$$
G_t^{\nu,1/2}(x,y) \simeq \bigg[ \frac{xy}t \wedge 1\bigg]^{\nu+1/2}
	\frac{1}{\sqrt{t}} \exp\bigg(- \frac{(x-y)^2}{4t} \bigg),
$$
uniformly in $x,y \in (0,1)$ and $0 < t \le T$. Moreover,
$$
G_t^{\nu,1/2}(x,y) \simeq (xy)^{\nu+1/2} \times
	\begin{cases}
		\exp\Big( -t\big(\z_{1}^{\nu,1/2}\big)^2\Big), & \;\; \textrm{if} \;\; \nu > -1/2, \\
		1, & \;\; \textrm{if} \;\; \nu =-1/2, \\
		\exp\Big( t\big(\z_{0}^{\nu,1/2}\big)^2\Big), & \;\; \textrm{if} \;\; \nu < -1/2,
	\end{cases}
$$
uniformly in $x,y \in (0,1)$ and $t \ge T$.
\end{theorem}

\begin{remark}
Assuming $H=1/2$, the cases $\nu = \pm 1/2$ are explicit, see \eqref{expl}.
One has $\z_n^{-1/2,1/2} = \pi n$, $\z_n^{1/2,1/2} = \pi(n-1/2)$ and
$$
\psi_n^{-1/2,1/2}(x) =
	\begin{cases}
		1, & n=0, \\
		\sqrt{2} \cos(\pi n x), & n \ge 1
	\end{cases}
	\qquad \textrm{and} \qquad
\psi_n^{1/2,1/2}(x) = \sin\big(\pi(n-1/2)x\big).
$$
In the cases just mentioned sharp estimates for the kernels
$G_t^{\pm 1/2,1/2}(x,y)$ are known, since those are the classical heat kernels on $(0,1)$ corresponding to
the standard Laplacian $\mathbb{L}^{\pm 1/2} = -\frac{d^2}{dx^2}$ with the Neumann boundary condition at the
right endpoint and either the Neumann ($\nu=-1/2$) or the Dirichlet ($\nu=1/2$) condition at the left endpoint.
\end{remark}

The proof of Theorem \ref{thm:main} requires some preparations.
Thus, in Section \ref{ssec:jac} we describe the Jacobi trigonometric Lebesgue measure setting scaled to $(0,1)$
and invoke the relevant sharp estimates for the associated heat kernel. In Section \ref{ssec:prep} we verify
coincidence of domains of $\widetilde{\mathbb{L}}^{\nu,1/2}$ and its counterpart in the Jacobi context.
Then in Section \ref{ssec:proof} we give the proof of Theorem \ref{thm:main}.
For various comments on the method of proving Theorem \ref{thm:main}, see Section \ref{ssec:com}.

\subsection{Jacobi trigonometric Lebesgue measure setting scaled to $(0,1)$} \label{ssec:jac}

This context will serve us as a source of results to be transferred to the Fourier-Dini framework.
Let $P_k^{\alpha,\beta}$, 
$k=0,1,2,\ldots$, be the classical Jacobi polynomials
with type parameters $\alpha,\beta>-1$, cf.\ \cite{Sz}. Define
$$
\Phi_k^{\alpha,\beta}(x) = C_k^{\alpha,\beta} 
		\Big(\sin\frac{\pi x}2\Big)^{\alpha+1/2} \Big(\cos\frac{\pi x}2\Big)^{\beta+1/2}
		P_k^{\alpha,\beta}\big(\cos\pi x\big), \qquad k \ge 0, \quad x \in (0,1),
$$
with
$$
C_k^{\alpha,\beta} = \bigg( \frac{\pi (2k+\alpha+\beta+1)\Gamma(k+\alpha+\beta+1)\Gamma(k+1)}
	{\Gamma(k+\alpha+1)\Gamma(k+\beta+1)} \bigg)^{1/2},
$$
where for $k=0$ and $\alpha+\beta=-1$ the product $(2k+\alpha+\beta+1)\Gamma(k+\alpha+\beta+1)$
must be replaced by $\Gamma(\alpha+\beta+2)$. Then the system $\{\Phi_k^{\alpha,\beta}: k\ge 0\}$
is an orthonormal basis in $L^2((0,1),dx)$. Moreover, each $\Phi_k^{\alpha,\beta}$ is an eigenfunction
of the Jacobi differential operator
$$
\mathbb{J}^{\alpha,\beta} = -\frac{d^2}{dx^2} - \frac{\pi^2(1/4-\alpha^2)}{4\sin^2(\pi x /2)}
	- \frac{\pi^2(1/4-\beta^2)}{4\cos^2(\pi x/2)},
$$
and one has
$$
\mathbb{J}^{\alpha,\beta} \Phi_k^{\alpha,\beta} = \Lambda_k^{\alpha,\beta} \Phi_k^{\alpha,\beta},
	 \qquad k \ge 0, \qquad \textrm{where} \qquad \Lambda_k^{\alpha,\beta} = 
	 \pi^2 \Big( k + \frac{\alpha+\beta+1}2\Big)^2.
$$
Thus, acting initially on $\spann\{\Phi_k^{\a,\b} : k \ge 0\}$,
$\mathbb{J}^{\alpha,\beta}$ has a natural self-adjoint extension on $L^2((0,1),dx)$ given by
$$
\widetilde{\mathbb{J}}^{\alpha,\beta} f = \sum_{k=0}^{\infty} \Lambda_k^{\alpha,\beta} 
	\big\langle f, \Phi_k^{\alpha,\beta} \big\rangle \Phi_k^{\alpha,\beta}
$$
on the domain
$$
\domain \widetilde{\mathbb{J}}^{\alpha,\beta} = \bigg\{ f \in L^2((0,1),dx) : \sum_{k=0}^{\infty}
	\big| \Lambda_k^{\alpha,\beta} \big\langle f, \Phi_k^{\alpha,\beta}\big\rangle \big|^2 < \infty \bigg\}.
$$

The semigroup generated by $-\widetilde{\mathbb{J}}^{\alpha,\beta}$ has the integral representation
$$
\exp\big(-t\widetilde{\mathbb{J}}^{\alpha,\beta}\big)f(x) = 
\int_0^1 {K}_t^{\alpha,\beta}(x,y)f(y)\, dy,
$$
where the Jacobi heat kernel is given by
$$
{K}_t^{\alpha,\beta}(x,y) = \sum_{k=0}^{\infty} \exp\big( -t\Lambda_k^{\alpha,\beta}\big)
	\Phi_k^{\alpha,\beta}(x) \Phi_k^{\alpha,\beta}(y), \qquad x,y \in (0,1), \quad t > 0.
$$

Sharp short-time estimates for $K_t^{\alpha,\beta}(x,y)$ are deduced from the results in \cite{NSS1,NSS2,NSS3},
see \cite[Theorem~A]{NSS3},
the relation between various Jacobi heat kernels given at the end of \cite[Section 2]{NoSj} and a simple scaling argument.
\begin{theorem}[{\cite{NSS1,NSS2,NSS3}}] \label{thm:jac}
Assume that $\a,\b > -1$. Given any $T > 0$, the estimates
$$
K_t^{\a,\b}(x,y) \simeq  \bigg[ \frac{xy}t \wedge 1\bigg]^{\alpha+1/2}
	\bigg[ \frac{(1-x)(1-y)}t \wedge 1 \bigg]^{\beta+1/2} \frac{1}{\sqrt{t}} \exp\bigg(
		- \frac{(x-y)^2}{4t} \bigg),
$$
hold uniformly in $x,y \in (0,1)$ and $0 < t \le T$.
\end{theorem}
The large-time sharp estimates for $K_t^{\a,\b}(x,y)$ are known as well, see e.g.\ \cite[Remark 7.3]{LN},
but we will not need them in the present paper.

\subsection{Preparatory results} \label{ssec:prep}

Similarly as in \cite{NR2}, denote
$$
F^{\nu}(x) := \mathbb{L}^{\nu} - \mathbb{J}^{\nu,-1/2}
	= \bigg(\frac{1}4 - \nu^2 \bigg) \bigg[ \frac{\pi^2}{4\sin^2\frac{\pi x}2} - \frac{1}{x^2}\bigg].
$$
This function is continuous and monotone on $[0,1]$, with the value at $x=0$ understood in the limiting sense.
Thus, the extreme values on $[0,1]$ are
$$
F^{\nu}(0) = \bigg(\frac{1}4 - \nu^2 \bigg) \frac{\pi^2}{12} \qquad \textrm{and} \qquad
F^{\nu}(1) = \bigg(\frac{1}4 - \nu^2 \bigg) \bigg(\frac{\pi^2}4 -1\bigg).
$$
In particular,
$$
|F^{\nu}(x)| \le |F^{\nu}(1)| = \Bigg|\frac{1}4 - \nu^2\bigg| \bigg(\frac{\pi^2}4-1\bigg), \qquad x \in [0,1].
$$
Note that $\pi^2/12 \approx 0.82$ and $\pi^2/4-1 \approx 1.47$. 

\begin{lemma} \label{lem:tech}
Let $\nu > -1$. Then
\begin{equation} \label{341}
\Big\langle \widetilde{\mathbb{L}}^{\nu,1/2} \psi_n^{\nu,1/2}, \Phi_k^{\nu,-1/2}\Big\rangle
 = \Big\langle \psi_n^{\nu,1/2}, \mathbb{L}^{\nu}\Phi_k^{\nu,-1/2}\Big\rangle, \qquad
	n \ge \chi_{\{\nu > -1/2\}}, \quad k \ge 0,
\end{equation}
where $\mathbb{L}^{\nu}$ is the Bessel differential operator \eqref{id8}. Similarly,
\begin{equation} \label{342}
\Big\langle \widetilde{\mathbb{J}}^{\nu,-1/2}\Phi_k^{\nu,-1/2}, \psi_n^{\nu,1/2}\Big\rangle
 = \Big\langle \Phi_k^{\nu,-1/2}, \mathbb{J}^{\nu,-1/2} \psi_n^{\nu,1/2}\Big\rangle, \qquad
	n \ge \chi_{\{\nu > -1/2\}}, \quad k \ge 0.
\end{equation}
\end{lemma}

\begin{proof}
It is enough to show \eqref{341} since then \eqref{342} is a simple consequence of the identity
$\mathbb{J}^{\nu,-1/2} = \mathbb{L}^{\nu} - F^{\nu}$ for differential operators.

We shall use the divergence form of $\mathbb{L}^{\nu}$, see \eqref{id8}, and integrate twice by parts.
Denote
$$
\mathcal{I} := \Big\langle \widetilde{\mathbb{L}}^{\nu,1/2}\psi_n^{\nu,1/2}, \Phi_k^{\nu,-1/2} \Big\rangle
	= \Big\langle {\mathbb{L}}^{\nu}\psi_n^{\nu,1/2}, \Phi_k^{\nu,-1/2} \Big\rangle
	= \int_0^1 \mathbb{L}^{\nu}\psi_n^{\nu,1/2}(x) \Phi_k^{\nu,-1/2}(x)\, dx.
$$
Clearly, the above integral converges since $\mathbb{L}^{\nu}\psi_n^{\nu,1/2}$ (see Proposition \ref{prop:eigen})
and $\Phi_k^{\nu,-1/2}$ are in $L^2((0,1),dx)$. Then we have
\begin{align*}
\mathcal{I} & = - \int_0^1 x^{-\nu-1/2} \frac{d}{dx} \Big( x^{2\nu+1} \frac{d}{dx} \big[ x^{-\nu-1/2}\psi_n^{\nu,1/2}(x)\big] \Big)
	\Phi_k^{\nu,-1/2}(x)\, dx \\
& = - x^{\nu+1/2} \frac{d}{dx} \big[ x^{-\nu-1/2} \psi_n^{\nu,1/2}(x)\big] \Phi_k^{\nu,-1/2}(x) \Big|_0^1 \\
& \qquad + \int_0^1 \frac{d}{dx}\big[ x^{-\nu-1/2}\psi_n^{\nu,1/2}(x)\big] \frac{d}{dx}\big[ x^{-\nu-1/2}\Phi_k^{\nu,-1/2}(x)\big]
	x^{2\nu+1}\, dx \\
& =  - x^{\nu+1/2} \frac{d}{dx} \big[ x^{-\nu-1/2} \psi_n^{\nu,1/2}(x)\big] \Phi_k^{\nu,-1/2}(x) \Big|_0^1 \\
& \qquad + x^{\nu+1/2} \psi_n^{\nu,1/2}(x) \frac{d}{dx}\big[ x^{-\nu-1/2} \Phi_k^{\nu,-1/2}(x)\big] \Big|_0^1 \\
& \qquad - \int_0^1 \psi_n^{\nu,1/2}(x) x^{-\nu-1/2} \frac{d}{dx} \Big( x^{2\nu+1} \frac{d}{dx} \big[ x^{-\nu-1/2}\Phi_k^{\nu,-1/2}(x)
	\big] \Big)\, dx \\
& =: \mathcal{I}_1 + \mathcal{I}_2 + \mathcal{I}_3.
\end{align*}
Here $\mathcal{I}_3 = \big\langle \psi_n^{\nu,1/2}, \mathbb{L}^{\nu}\Phi_k^{\nu,-1/2} \big\rangle$, see \eqref{id8}, so it
remains to show that $\mathcal{I}_1 + \mathcal{I}_2 = 0$.

A direct computation with the aid of \eqref{id5} shows that for $n \ge 1$
$$
\mathcal{I}_1 = c_n^{\nu,1/2} \z_n^{\nu,1/2} \sqrt{x} J_{\nu+1}\big(\z_n^{\nu,1/2}x\big) \, C_k^{\nu,-1/2}
	\Big(\sin\frac{\pi x}{2}\Big)^{\nu+1/2} P_k^{\nu,-1/2}(\cos\pi x) \Big|_0^1,
$$
$\mathcal{I}_1 = 0$ in case $n=0$ and $\nu=-1/2$, and when $n=0$ and $\nu < -1/2$ one has
$$
\mathcal{I}_1 = - c_0^{\nu,1/2} \z_0^{\nu,1/2} \sqrt{x} I_{\nu+1}\big(\z_0^{\nu,1/2}x\big)\, C_k^{\nu,-1/2}
	\Big(\sin\frac{\pi x}{2}\Big)^{\nu+1/2} P_k^{\nu,-1/2}(\cos\pi x) \Big|_0^1.
$$
Using now \eqref{as1} we see that at $x=0^+$ the expression defining $\mathcal{I}_1$ vanishes, so only the value at $x=1$
matters and we get
\begin{equation} \label{I1f}
\mathcal{I}_1 = 
	\begin{cases}
		c_n^{\nu,1/2} \z_n^{\nu,1/2} J_{\nu+1}\big(\z_n^{\nu,1/2}\big)\, C_k^{\nu,-1/2} P_k^{\nu,-1/2}(-1), & n \ge 1, \\
		0, & n=0 \;\; \textrm{and} \;\; \nu=-1/2, \\
		- c_0^{\nu,1/2} \z_0^{\nu,1/2} I_{\nu+1}\big(\z_0^{\nu,1/2}\big)\, C_k^{\nu,-1/2} P_k^{\nu,-1/2}(-1), & n = 0 \;\;
			\textrm{and} \;\; \nu < -1/2.
	\end{cases} 
\end{equation}

To evaluate $\mathcal{I}_2$ we use the differentiation rule for Jacobi polynomials (cf.\ \cite[(4.21.7)]{Sz})
\begin{equation}\label{eq:jacder}
\frac{d}{du} P_k^{\a,\b}(u) = \frac{1}2 (k+\a+\b+1) P_{k-1}^{\a+1,\b+1}(u), \qquad k \ge 0,
\end{equation}
with the convention that $P^{\a,\b}_{-1} \equiv 0$. We get for $n \ge 1$
\begin{align*}
\mathcal{I}_2 & = 
c_n^{\nu,1/2} x^{\nu+1} J_{\nu}\big(\z_n^{\nu,1/2}x \big)\, C_k^{\nu,-1/2} \Bigg[ \Big(\nu+\frac{1}2\Big)
	\bigg( \frac{\sin\frac{\pi x}2}{x}\bigg)^{\nu-1/2} x \, \bigg\{ \frac{\frac{\pi x}2 \cos\frac{\pi x}2 - \sin\frac{\pi x}2}{x^3}\bigg\}
		P_k^{\nu,-1/2}(\cos\pi x) \\
& \qquad - \bigg( \frac{\sin\frac{\pi x}2}{x}\bigg)^{\nu+1/2} \frac{\pi (k+\nu+1/2)}{2} \sin(\pi x) P_{k-1}^{\nu+1,1/2}(\cos\pi x)
	\Bigg] \Bigg|_0^1,
\end{align*}
$\mathcal{I}_2 = 0$ in case $n=0$ and $\nu = -1/2$, and when $n=0$ and $\nu < -1/2$ the quantity $\mathcal{I}_2$
is given by the above expression with $n=0$ and $J_{\nu}$ replaced by $I_{\nu}$.
Using \eqref{as1} and observing that the quotient in curly brackets$^{\natural}$
stays bounded near $x=0$, we see that the expression defining $\mathcal{I}_2$ vanishes at $x=0^+$.
Consequently, taking the value at $x=1$,
\footnote{$\natural$ It is perhaps curious to observe that the expression in curly brackets is equal to
$-(\pi/2)^2 x^{-3/2}J_{3/2}({\pi x}/2)$.}
\begin{equation} \label{I2f}
\mathcal{I}_2 =
	\begin{cases}
		- c_n^{\nu,1/2} J_{\nu}\big(\z_n^{\nu,1/2}\big)\, (\nu+1/2) C_k^{\nu,-1/2} P_k^{\nu,-1/2}(-1), & n \ge 1, \\
		0, & n=0 \;\; \textrm{and} \;\; \nu=-1/2, \\
		- c_0^{\nu,1/2} I_{\nu}\big(\z_0^{\nu,1/2}\big)\, (\nu+1/2) C_k^{\nu,-1/2}
			P_k^{\nu,-1/2}(-1), & n=0 \;\; \textrm{and} \;\; \nu < -1/2.
	\end{cases}
\end{equation}

Adding \eqref{I1f} and \eqref{I2f} we obtain
$$
\mathcal{I}_1 + \mathcal{I}_2 =
	\begin{cases}
		-c_n^{\nu,1/2} C_k^{\nu,-1/2} P_k^{\nu,-1/2}(-1)\big[ (\nu+1/2) J_{\nu}\big(\z_n^{\nu,1/2}\big)
			- \z_n^{\nu,1/2}J_{\nu+1}\big(\z_n^{\nu,1/2}\big)\big], & n \ge 1, \\
		0, & n= 0,\;\; \nu = -1/2, \\
		-c_0^{\nu,1/2} C_k^{\nu,-1/2} P_k^{\nu,-1/2}(-1)\big[ (\nu+1/2) I_{\nu}\big(\z_0^{\nu,1/2}\big)
			+ \z_0^{\nu,1/2}I_{\nu+1}\big(\z_0^{\nu,1/2}\big)\big],&  n = 0,\;\; \nu < -1/2.
	\end{cases}
$$
In view of \eqref{id6} and \eqref{id66},
$$
\mathcal{I}_1 + \mathcal{I}_2 =
- c_n^{\nu,1/2} C_k^{\nu,-1/2} P_k^{\nu,-1/2}(-1) \times
	\begin{cases}
		J_{\nu,1/2}\big(\z_n^{\nu,1/2}\big), & n \ge 1, \\
		0, & n=0 \;\; \textrm{and} \;\; \nu=-1/2, \\
		I_{\nu,1/2}\big(\z_0^{\nu,1/2}\big), & n=0 \;\; \textrm{and} \;\; \nu < -1/2.
	\end{cases}
$$
It follows that $\mathcal{I}_1 + \mathcal{I}_2 = 0$, which finishes the proof.
\end{proof}

\begin{lemma} \label{lem:tech2}
For each $\nu > -1$,
$$
\domain \widetilde{\mathbb{L}}^{\nu,1/2} = \domain \widetilde{\mathbb{J}}^{\nu,-1/2}.
$$
\end{lemma}

\begin{proof}
The reasoning is essentially the same as in the proof of \cite[Theorem 3.1]{NR2}$^\clubsuit$.
We will verify the inclusion $\domain\widetilde{\mathbb{J}}^{\nu,-1/2} \subset \domain\widetilde{\mathbb{L}}^{\nu,1/2}$.
The other inclusion follows by the same arguments, with the aid of \eqref{342} in Lemma \ref{lem:tech}.
\footnote{$\clubsuit$
In the proof of \cite[Theorem 3.1]{NR2} an argument based on self-adjointness used to conclude coincidence of
domains from their inclusion appears to be invalid. Fortunately, it can be easily replaced by an argument
given in the proof of Lemma~\ref{lem:tech2} (this requires also a slight extension of \cite[Lemma 3.2]{NR2}, see Lemma~\ref{lem:tech}).
Analogous remarks pertain to \cite[Lemma 7.6]{LN} and \cite[Lemma 7.5]{LN}.
}

Pick an arbitrary $f \in \domain\widetilde{\mathbb{J}}^{\nu,-1/2}$. We need to show that
$$
S:= \sum_{n \ge \chi_{\{\nu > -1/2\}}} \Big| \big(\z_n^{\nu,1/2}\big)^2 \big\langle f, \psi_n^{\nu,1/2}\big\rangle\Big|^2 < \infty.
$$

Observe that
$$
S = \sum_{n \ge \chi_{\{\nu > -1/2\}}} \Big| \Big\langle f, \widetilde{\mathbb{L}}^{\nu,1/2}\psi_n^{\nu,1/2}\Big\rangle\Big|^2
	= \sum_{n \ge \chi_{\{\nu > -1/2\}}} \bigg| \sum_{k=0}^{\infty} \big\langle f, \Phi_k^{\nu,-1/2}\big\rangle
		\Big\langle \Phi_k^{\nu,-1/2}, \widetilde{\mathbb{L}}^{\nu,1/2}\psi_n^{\nu,1/2}\Big\rangle\bigg|^2.
$$
Using now \eqref{341} in Lemma \ref{lem:tech} and recalling that $\mathbb{L}^{\nu}= \mathbb{J}^{\nu,-1/2} + F^{\nu}$, we can write
\begin{align*}
S & = \sum_{n \ge \chi_{\{\nu > -1/2\}}} \bigg| \sum_{k=0}^{\infty} \big\langle f, \Phi_k^{\nu,-1/2}\big\rangle
	\Big\langle \big(\mathbb{J}^{\nu,-1/2} + F^{\nu}\big) \Phi_k^{\nu,-1/2}, \psi_n^{\nu,1/2}\Big\rangle \bigg|^2 \\
& = \sum_{n \ge \chi_{\{\nu > -1/2\}}} \Bigg| \Bigg\langle \sum_{k=0}^{\infty} \Lambda_k^{\nu,-1/2} \big\langle
	f,\Phi_k^{\nu,-1/2}\big\rangle \Phi_k^{\nu,-1/2}, \psi_n^{\nu,1/2} \Bigg\rangle
	+ \Bigg\langle \sum_{k=0}^{\infty} \big\langle f, \Phi_k^{\nu,-1/2}\big\rangle \Phi_k^{\nu,-1/2}, \psi_n^{\nu,1/2}F^{\nu}
		\Bigg\rangle \Bigg|^2 \\
& = \sum_{n \ge \chi_{\{\nu > -1/2\}}} \Big| \Big\langle \widetilde{\mathbb{J}}^{\nu,-1/2}f, \psi_n^{\nu,1/2}\Big\rangle
	+ \big\langle f F^{\nu}, \psi_n^{\nu,1/2}\big\rangle \Big|^2.
\end{align*}
This, together with Parseval's identity, implies
$$
S \le 2 \big\|\widetilde{\mathbb{J}}^{\nu,-1/2}f\big\|^2_{L^2((0,1),dx)}
	+ 2 \big\| f F^{\nu}\big\|^2_{L^2((0,1),dx)}.
$$
Here the right-hand side is finite (recall that $F^{\nu}$ is bounded on $(0,1)$) and the conclusion follows.
\end{proof}

\subsection{{Proof of Theorem \ref{thm:main}}} \label{ssec:proof}
To begin with, we show that the Fourier-Dini heat kernel $G_t^{\nu,1/2}(x,y)$ is related to the Jacobi heat kernel
$K_t^{\nu,-1/2}(x,y)$ as follows.
\begin{proposition} \label{prop:GKrel}
For $\nu \in (-1,-1/2) \cup (1/2,\infty)$, one has
$$
e^{t |F^{\nu}(0)|} K_t^{\nu,-1/2}(x,y) \le G_t^{\nu,1/2}(x,y) \le e^{t |F^{\nu}(1)|} K_t^{\nu,-1/2}(x,y),
$$
while for $\nu \in [-1/2,1/2]$
$$
e^{-t F^{\nu}(1)} K_t^{\nu,-1/2}(x,y) \le G_t^{\nu,1/2}(x,y) \le e^{-t F^{\nu}(0)} K_t^{\nu,-1/2}(x,y),
$$
all inequalities for $x,y \in (0,1)$ and $t > 0$.
\end{proposition}
Notice that, in view of Proposition \ref{prop:GKrel} and the fact that $F^{\pm 1/2}(x) \equiv 0$,
one has $G_t^{\pm 1/2,1/2}(x,y) \equiv K_t^{\pm 1/2, -1/2}(x,y)$; these identities can of course be verified directly.

\begin{proof}[{Proof of Proposition \ref{prop:GKrel}}]
We follow the idea from \cite{NR2} and employ the Trotter product formula, cf.\ \cite[Chap.\,VIII, Sec.\,8]{RS},
which can be formulated as follows.

Let $A$ and $B$ be (not necessarily bounded) self-adjoint operators on a Hilbert space $\mathcal{H}$.
If $A+B$ is essentially self-adjoint on $\domain A \cap \domain B$, and $A$ and $B$ are both bounded from below, then
$$
\exp\big( -t(A+B)\big)h = \lim_{m \to \infty}\big[ \exp(-tA/m) \exp(-tB/m)\big]^m h, \qquad h \in \mathcal{H}, \quad t \ge 0.
$$

For our purpose we let $\mathcal{H} = L^2((0,1),dx)$, $A = \widetilde{\mathbb{J}}^{\nu,-1/2}$ and $B$ be the multiplication
operator by $F^{\nu}$. Recall that $F^{\nu}$ is a continuous bounded function on $[0,1]$ with extreme values at the endpoints
of this interval. In particular, $B$ is bounded on $\mathcal{H}$ and Lemma \ref{lem:tech2} implies that
$A+B = \widetilde{\mathbb{L}}^{\nu,1/2}$. Further, we note that $\exp(-t B/m)$ is the multiplication operator by the function
$x \mapsto \exp(-t F^{\nu}(x)/m)$ and the Jacobi semigroup $\{\exp(-t\widetilde{\mathbb{J}}^{\nu,-1/2})\}$,
being positive (in the lattice sense), preserves inequalities between functions in $L^2((0,1),dx)$.
Thus, estimating $x \to \exp(-t F^{\nu}(x)/m)$ by its extrema on $[0,1]$, from the Trotter formula we conclude that
$$
\exp\Big( -t\max_{x \in [0,1]} F^{\nu}(x)\Big) \exp\big(-t\widetilde{\mathbb{J}}^{\nu,-1/2}\big)f \le T_t^{\nu,1/2}f
\le \exp\Big( -t\min_{x \in [0,1]} F^{\nu}(x)\Big) \exp\big(-t\widetilde{\mathbb{J}}^{\nu,-1/2}\big)f
$$
for all non-negative $f \in L^2((0,1),dx)$ and $t \ge 0$.
Since the corresponding heat kernels are continuous functions of their arguments, we get
$$
\exp\Big( -t\max_{x \in [0,1]} F^{\nu}(x)\Big) K_t^{\nu,-1/2}(x,y) \le G_t^{\nu,1/2}(x,y)
\le \exp\Big( -t\min_{x \in [0,1]} F^{\nu}(x)\Big) K_t^{\nu,-1/2}(x,y)
$$
for $x,y \in (0,1)$ and $t > 0$. The asserted bounds follow.
\end{proof}

Another ingredient we need are sharp large-time bounds for $G_t^{\nu,1/2}(x,y)$.
We state this result in greater generality for $G_t^{\nu,H}(x,y)$ with an arbitrary $H \in \mathbb{R}$, since it might
be of independent interest.
\begin{proposition} \label{prop:tlarge}
Let $\nu > -1$ and $H \in \mathbb{R}$ be fixed.
Then there exists $T > 0$ such that
$$
G_t^{\nu,H}(x,y) \simeq (xy)^{\nu+1/2} \times
	\begin{cases}
		\exp\Big( -t\big(\z_{1}^{\nu,H}\big)^2\Big), & \;\; \textrm{if} \;\; H + \nu > 0, \\
		1, & \;\; \textrm{if} \;\; H + \nu = 0, \\
		\exp\Big( t\big(\z_{0}^{\nu,H}\big)^2\Big), & \;\; \textrm{if} \;\; H + \nu < 0,
	\end{cases}
$$
uniformly in $x,y \in (0,1)$ and $t \ge T$.
\end{proposition}

\begin{proof}
Combine the series representation \eqref{serG} with the asymptotics \eqref{azer}, \eqref{as1} and \eqref{as2}.
The details are elementary and standard, thus left to the reader.
\end{proof}

\begin{proof}[{Proof of Theorem \ref{thm:main}}]
The theorem is a direct consequence of Proposition \ref{prop:GKrel}, the short/bounded-time bounds for the Jacobi
heat kernel stated in Theorem \ref{thm:jac} and the large-time bounds from Proposition~\ref{prop:tlarge}.
\end{proof}

\subsection{Comments on the proof method} \label{ssec:com}

The idea of transference of short-time sharp heat kernel bounds from the Jacobi to the Fourier-Bessel setting was presented in
\cite{NR2}. Note that the Lebesgue measure Fourier-Bessel context can be formally interpreted
as the Fourier-Dini context with $H=\infty$, cf.\ \cite[Chap.\,XVIII, Sec.\,18{$\cdot$}11]{watson}.
In the present paper we further explore the idea to cover the case $H=1/2$.
As explained below, these two cases fully exhaust the applicability of the method within the Fourier-Bessel
and Fourier-Dini frameworks. Thus, treatment of other $H$ requires a different strategy.
Notably, there are expansions related to the Fourier-Bessel and Fourier-Dini contexts where the method still applies,
see \cite[Section 7]{LN}. In those cases, the transference source is certain expansions related to the Jacobi context.

There are two fundamental points that make the method applicable: boundedness of the difference between the
corresponding `Laplacians' (cf.\ properties of the function $F^{\nu}$) and coincidence of their domains (see Lemma \ref{lem:tech2}).

Consider $\a,\b,\nu > -1$. It is clear that
\begin{equation} \label{LJdiff}
\mathbb{L}^{\nu} - \mathbb{J}^{\a,\b} = \frac{\pi^2(1/4-\a^2)}{4\sin^2(\pi x /2)}
	+ \frac{\pi^2(1/4-\b^2)}{4\cos^2(\pi x /2)} - \frac{1/4-\nu^2}{x^2}
\end{equation}
is bounded on $(0,1)$ if and only if $\b = \pm 1/2$ and $\a = \pm \nu$.
On the other hand, the following holds.
\begin{theorem} \label{thm:dom}
Let $\a,\b,\nu > -1$ and $H \in \mathbb{R}$. Then the coincidence of domains
$$
\domain \widetilde{\mathbb{L}}^{\nu,H} = \domain \widetilde{\mathbb{J}}^{\a,\b}
$$
holds if and only if
$$
H=1/2 \qquad \textrm{and} \qquad \b=-1/2 \qquad \textrm{and} \qquad \big[\; \nu = \a \quad \textrm{or} \quad \a,\nu > 1 \; \big].
$$
\end{theorem}

Altogether, this shows that the necessary and sufficient condition to have simultaneously boundedness of
$\mathbb{L}^{\nu} - \mathbb{J}^{\a,\b}$ and the coincidence of domains of $\widetilde{\mathbb{L}}^{\nu,H}$
and $\widetilde{\mathbb{J}}^{\a,\b}$ is $H=1/2$, $\b=-1/2$, $\nu=\a$. This is precisely the situation treated in the paper.

Theorem \ref{thm:dom} is actually stronger than necessary for our purpose, but the full characterization we
state seems to be of independent interest.
The proof of Theorem \ref{thm:dom} is presented in Section \ref{sec:dom} below. It is somewhat lengthy and technical,
therefore we postpone it to the last section to avoid possible confusion with the main line of the paper.

\section{Applications} \label{sec:app}

This section highlights several interesting applications of Theorem \ref{thm:main}.
To this end we always consider the specific $H=1/2$.
For terminology concerning strong, weak and restricted weak type boundedness/estimates we refer to
\cite[Chap.\,V]{SW}.

\subsection{Maximal operators and boundary convergence of the Fourier-Dini semigroup}
In what follows we assume that each $T_t^{\nu,1/2}$, $t > 0$, is an integral operator associated with the kernel
$G_t^{\nu,1/2}(x,y)$ on a natural domain consisting of all functions for which the defining integral makes sense
for a.e.\ $x \in (0,1)$.

Consider the maximal operators
$$
T_*^{\nu,1/2}f = \sup_{t > 0}\big| T_t^{\nu,1/2}f \big| \qquad \textrm{and} \qquad
T_{*,\textrm{loc}}^{\nu,1/2}f = \sup_{0 < t \le 1}\big| T_t^{\nu,1/2}f \big|.
$$
Observe that obviously $T_{*,\textrm{loc}}^{\nu,1/2}f \le T_*^{\nu,1/2}f$, and that $T_*^{\nu,1/2}f$ is infinite
(cf.\ the large-time bounds in Theorem \ref{thm:main}) for $-1 < \nu < -1/2$ and, say, $f>0$.

As a consequence of Theorem \ref{thm:main} we will show the following.
\begin{theorem} \label{thm:max}
Let $1 \le p \le \infty$.
\begin{itemize}
\item[(a)]
If $\nu \ge -1/2$, then $T_*^{\nu,1/2}$ is bounded on $L^p((0,1),dx)$ for $p > 1$ and satisfies the weak type $(1,1)$ estimate.
\item[(b)]
If $-1 < \nu < -1/2$, then $T_{*,\textrm{loc}}^{\nu,1/2}$ is bounded on $L^p((0,1),dx)$ for $p_0 < p < p_1$,
satisfies the weak type $(p_1,p_1)$ estimate and the restricted weak type $(p_0,p_0)$ estimate, where
$p_0 := 1/(\nu+3/2)$ and $p_1:= -1/(\nu+1/2)$.
\end{itemize}
\end{theorem}

\begin{proof}
We shall relate our maximal operators to the maximal operator of the Bessel semigroup corresponding to the measure space
$((0,\infty),dx)$ and use known results in the Bessel setting for which we use notation consistent with \cite{BHNV}.

The Bessel heat kernel $\widetilde{W}_t^{\lambda}(x,y)$ satisfies the sharp bounds
$$
\widetilde{W}_{t}^{\lambda}(x,y) \simeq \bigg[ \frac{xy}t \wedge 1\bigg]^{\lambda} \frac{1}{\sqrt{t}}
	\exp\bigg( -\frac{(x-y)^2}{4t}\bigg), \qquad x,y,t>0,
$$
which are a direct consequence of a closed formula for the kernel in terms of $I_{\lambda-1/2}$ (see e.g.\ \cite{BHNV})
and basic asymptotics \eqref{as1}, \eqref{as2} for the modified Bessel function.
Thus, in view of Theorem \ref{thm:main}, we have
$$
G_t^{\nu,1/2}(x,y) \simeq \widetilde{W}_t^{\nu+1/2}(x,y), \qquad x,y \in (0,1), \quad 0 < t \le 1.
$$
Moreover, for $\nu > -1/2$,
$$
G_t^{\nu,1/2}(x,y) \lesssim \widetilde{W}_t^{\nu+1/2}(x,y), \qquad x,y \in (0,1), \quad t > 0,
$$
and for $\nu=-1/2$ we have the simple large-time behavior $G_t^{\nu,1/2}(x,y) \simeq 1$ for $x,y \in (0,1)$ and $t \ge 1$.

Passing to maximal operators, we can restrict attention to $f \ge 0$. Then we see that $T_{*}^{\nu,1/2}f$ is controlled by
$\widetilde{W}_{*}^{\nu+1/2}f$ when $\nu \ge -1/2$ (to be precise, by $\widetilde{W}_{*}^{\nu+1/2}f$ plus a trivial operator
in case $\nu=-1/2$). Likewise, $T_{*,\textrm{loc}}^{\nu}f$ is controlled by $\widetilde{W}_{*}^{\nu+1/2}f$ for each $\nu > -1$.

Now the mapping properties stated in the theorem can be concluded from analogous results for $\widetilde{W}_*^{\nu+1/2}$
that can be found in \cite[Remark 3.2]{BHNV}.
\end{proof}

One can show that Theorem \ref{thm:max} is optimal in the sense that the mapping properties stated are best possible among
strong, weak and restricted weak type estimates. Also, from the results in \cite{BHNV} one could infer more general weighted
mapping properties of $T_*^{\nu,1/2}$ and $T_{*,\textrm{loc}}^{\nu,1/2}$,
but we shall not pursue these matters here.

As a natural consequence of Theorem \ref{thm:max} we get the following.
\begin{corollary} \label{cor:conv}
The boundary convergence
$$
T_t^{\nu,1/2}f(x) \longrightarrow f(x) \quad \textrm{a.e.}\;\; x \in (0,1) \quad \textrm{as} \quad t \to 0^+
$$
holds for $f \in L^1((0,1),dx)$ in case $\nu \ge -1/2$ and for $f \in L^p((0,1),dx)$, $p > p_0$, in case $-1 < \nu < -1/2$;
here $p_0$ is the same as in Theorem \ref{thm:max}.
\end{corollary}

\begin{proof}
The arguments are standard, see e.g.\ \cite[Chap.\,2, Sec.\,2]{Duo}.
Notice that $L^1((0,1),dx)$ is the biggest space among $L^p((0,1),dx)$, $p \ge 1$.
The crucial ingredients of the proof are the maximal Theorem \ref{thm:max} and the fact that the Fourier-Dini system is linearly dense
in $L^2((0,1),dx)$, hence also in the $L^p$ spaces in question.
Further details are left to the reader.
\end{proof}

\subsection{Sharp bounds for the Fourier-Dini-Poisson kernel}
Denote by $H_t^{\nu,1/2}(x,y)$ the integral kernel of the Fourier-Dini-Poisson semigroup, i.e.\ the semigroup generated by
$-(\widetilde{\mathbb{L}}^{\nu,1/2})^{1/2}$. Here we assume that $\nu \ge -1/2$,
since otherwise $\widetilde{\mathbb{L}}^{\nu,1/2}$ is not non-negative.
For $\nu \in (-1,-1/2)$, and more generally for $\nu > -1$, we consider the kernel $H_{t,\to}^{\nu,1/2}(x,y)$ of the semigroup
generated by $-(d_{\nu}^2 \id + \widetilde{\mathbb{L}}^{\nu,1/2})^{1/2}$, where $d_{\nu} \ge \z_0^{\nu,1/2}$ so that
the shifted Fourier-Dini Laplacian $d^2_{\nu} \id + \widetilde{\mathbb{L}}^{\nu,1/2}$ is non-negative.
For instance, $d_{\nu}=1$ is a good choice since $\z_0^{\nu,1/2} < 1/2$ for $\nu \in (-1,-1/2)$;
see Proposition \ref{prop:eest} below. When $\nu \ge -1/2$, one of course can take $d_{\nu}=0$ and then
$H_{t,\to}^{\nu,1/2}(x,y) = H_{t}^{\nu,1/2}(x,y)$.

The following sharp bounds can be obtained with the aid of Theorem \ref{thm:main}.
\begin{theorem} \label{thm:poisson}
Assume that $\nu > -1$. Given any $T > 0$,
$$
H_{t,\to}^{\nu,1/2}(x,y) \simeq \bigg( \frac{\sqrt{xy}}{t+x+y}\bigg)^{2\nu+1} \frac{t}{t^2+(x-y)^2}
$$
uniformly in $x,y \in (0,1)$ and $0 < t \le T$. The same bounds hold for $H_{t}^{\nu,1/2}(x,y)$ when $\nu \ge -1/2$.
Moreover,
$$
H_{t,\to}^{\nu,1/2}(x,y) \simeq (xy)^{\nu+1/2} \times
	\begin{cases}
		\exp\Big( -t\sqrt{\big(\z_{1}^{\nu,H}\big)^2+d_{\nu}^2}\;\Big), & \;\; \textrm{if} \;\; \nu > -1/2, \\
		\exp(-t d_{\nu}), & \;\; \textrm{if} \;\; \nu = -1/2, \\
		\exp\Big( -t\sqrt{d_{\nu}^2-\big(\z_{0}^{\nu,H}\big)^2}\;\Big), & \;\; \textrm{if} \;\; \nu < -1/2,
	\end{cases}
$$
uniformly, in $x,y \in (0,1)$ and $t \ge T$.
\end{theorem}

\begin{proof}
Clearly, the kernels of the semigroups generated by $-\widetilde{\mathbb{L}}^{\nu,1/2}$ and
$-(d_{\nu}^2 \id + \widetilde{\mathbb{L}}^{\nu,1/2})$ have the same short-time bounds stated in Theorem \ref{thm:main}.
Further, the large-time bounds of Theorem \ref{thm:main} instantly imply analogous bounds for the kernel of
$\exp(-(d_{\nu}^2 \id + \widetilde{\mathbb{L}}^{\nu,1/2}))$, where the difference is that one has to multiply by the factor
$\exp(-t d_{\nu}^2)$.

Now we are in a position where the arguments from the proof of \cite[Theorem 2.5]{NR3} apply, since the short-time bounds
of Theorem \ref{thm:main} are exactly the same as the analogous bounds for the Jacobi heat kernel
with $\a=\nu$ and $\b = -1/2$, cf.\ Theorem \ref{thm:jac}.
Roughly speaking, the idea relies on employing known sharp bounds for the Jacobi-Poisson kernel obtained in \cite[Theorem A.1]{NoSj}
and \cite[Theorem 6.1]{NSS}.
We omit the details.
\end{proof}

We now return to estimating the size of $\z_0^{\nu,1/2}$.
\begin{proposition} \label{prop:eest}
One has the bound
$$
\z_0^{\nu,1/2} < 1/2, \qquad \nu \in (-1,-1/2).
$$
\end{proposition}

\begin{proof}
Let $\nu \in (-1,-1/2)$ and observe that $\z_0^{\nu,1/2}$ is the only strictly positive root of the equation
$$
\frac{I_{\nu+1}(x)}{I_{\nu}(x)} = - \frac{\nu+1/2}{x}.
$$
By \cite[Theorem 3(a)]{nasell}, see also \cite[p.\,11, formula for $L_{\nu,0,1}$]{nasell}, the left-hand side here can be
estimated from below
$$
\frac{I_{\nu+1}(x)}{I_{\nu}(x)} > \frac{2(\nu+2)x}{4(\nu+1)(\nu+2)+x^2} =: U(x), \qquad x > 0.
$$
Then the (first) positive root of the equation $U(x) = -(\nu+1/2)/x$, call it $x_0^{\nu}$, is bigger than $\z_0^{\nu,1/2}$.
Solving the last equation for positive roots we get
$$
x_0^{\nu} = \frac{2}{3}\sqrt{- \frac{6\nu^3 + 21\nu^2 + 21\nu + 6}{2\nu+3}}.
$$
Using this explicit expression it is elementary to check that $x_0^{\nu} < 1/2$ for $\nu$ under consideration.
Since $\z_0^{\nu,1/2} < x_0^{\nu}$, the conclusion follows.
\end{proof}

\subsection{Sharp bounds for Fourier-Dini potential kernels}
Let $\sigma > 0$.
Denote by $A_{\sigma}^{\nu,1/2}(x,y)$ the integral kernel associated with the Riesz potential operator
$(\widetilde{\mathbb{L}}^{\nu,1/2})^{-\sigma}$. Here the negative power makes sense only for $\nu > -1/2$ since otherwise
$\widetilde{\mathbb{L}}^{\nu,1/2}$ has an eigenvalue that is not strictly positive.
Further, denote by $B_{\sigma}^{\nu,1/2}(x,y)$ the integral kernel associated with the Bessel potential
$(\id + \widetilde{\mathbb{L}}^{\nu,1/2})^{-\sigma}$, which is well defined for all $\nu > -1$.
Both the potential operators are defined initially in $L^2((0,1),dx)$ via the spectral theorem.
The potential kernels express as
$$
A_{\sigma}^{\nu,1/2}(x,y) = \frac{1}{\Gamma(2\sigma)}\int_0^{\infty} H_t^{\nu,1/2}(x,y) t^{2\sigma-1}\, dt, \qquad
B_{\sigma}^{\nu,1/2}(x,y) = \frac{1}{\Gamma(2\sigma)}\int_0^{\infty} H_{t,\to}^{\nu,1/2}(x,y) t^{2\sigma-1}\, dt,
$$
where $x,y \in (0,1)$ and $H_{t,\to}^{\nu,1/2}(x,y)$ comes with the choice $d_{\nu}=1$; see e.g.\ \cite[(2)]{NR3}.

The following sharp estimates hold.
\begin{theorem} \label{thm:potk}
Let $\nu > -1$ and $\sigma > 0$ be fixed. Then
\begin{align*}
(xy)^{-\nu-1/2} B_{\sigma}^{\nu,1/2}(x,y) & \simeq 1 + \chi_{\{\sigma=\nu+1\}} \log\frac{2}{x+y}
	+ \chi_{\{\sigma = 1/2\}} \log\frac{2}{2-x-y} \\
& \quad + (x+y)^{2\sigma-2(\nu+1)} \times
	\begin{cases}
		(2-x-y)^{2\sigma-1}, & \;\; \textrm{if} \;\; \sigma > 1/2, \\
		\log\frac{(x+y)(2-x-y)}{|x-y|},& \;\; \textrm{if} \;\; \sigma=1/2, \\
		\big( \frac{x+y}{|x-y|} \big)^{1-2\sigma}, & \;\; \textrm{if} \;\; \sigma < 1/2,
	\end{cases}
\end{align*}
uniformly in $x,y \in (0,1)$. The same holds with $B_{\sigma}^{\nu,1/2}$ replaced by $A_{\sigma}^{\nu,1/2}$
provided that $\nu > -1/2$.
\end{theorem}

\begin{proof}
From Theorem \ref{thm:poisson}, or in fact from its proof, we know that the short-time behavior of the
Fourier-Dini-Poisson kernels is exactly
the same as in the Jacobi setting from Section \ref{ssec:jac} with parameters $\alpha=\nu$ and $\beta=-1/2$.
This, together with the large-time bounds from Theorem \ref{thm:poisson} (the exact values of positive constants in the exponential
factors are irrelevant), implies that the sharp bounds for $A_{\sigma}^{\nu,1/2}(x,y)$ and $B_{\sigma}^{\nu,1/2}(x,y)$
are the same as in the Jacobi framework with parameters $\alpha=\nu$ and $\beta=-1/2$.
Therefore, using \cite[Theorem 2.2]{NR3}, see also \cite[(5)]{NR3}, we arrive at the desired conclusion.
\end{proof}

\subsection{Sharp $L^p-L^q$ mapping properties of Fourier-Dini potential operators}
For $\nu > -1/2$ and for $\nu > -1$ consider, respectively, the integral operators
$$
A_{\sigma}^{\nu,1/2}f(x) = \int_0^1 A_{\sigma}^{\nu,1/2}(x,y) f(y)\, dy \qquad \textrm{and} \qquad
B_{\sigma}^{\nu,1/2}f(x) = \int_0^1 B_{\sigma}^{\nu,1/2}(x,y) f(y)\, dy
$$
on their natural domains $\domain A_{\sigma}^{\nu,1/2}$ and $\domain B_{\sigma}^{\nu,1/2}$ consisting of all functions $f$ on $(0,1)$
for which the respective integrals converge for a.e.\ $x \in (0,1)$. Using Theorem \ref{thm:potk} it is
straightforward to check that, when $\nu > -1/2$, $\domain A_{\sigma}^{\nu,1/2}$ contains all $L^p((0,1),dx)$ spaces,
$1\le p \le \infty$. The same is true about $\domain B_{\sigma}^{\nu,1/2}$ if $\nu \ge -1/2$. If $\nu \in (-1,-1/2)$, then
$L^p((0,1),dx) \subset \domain B_{\sigma}^{\nu,1/2}$ provided that $1/p < \nu + 3/2$.
Note that $L^2((0,1),dx)$ is always contained in both domains, and in $L^2((0,1),dx)$ the integral operators coincide with the
corresponding negative powers of $\widetilde{\mathbb{L}}^{\nu,1/2}$ and $\id + \widetilde{\mathbb{L}}^{\nu,1/2}$, respectively,
defined spectrally.

The following result gives a complete and sharp description of $L^p-L^q$ mapping properties of $A_{\sigma}^{\nu,1/2}$
and $B_{\sigma}^{\nu,1/2}$.
\begin{theorem} \label{thm:pot}
Let $\nu > -1$, $\sigma > 0$ and $1 \le p,q \le \infty$.
Then $B_{\sigma}^{\nu,1/2}$ has the following mapping properties with respect to the measure space $((0,1),dx)$.
\begin{itemize}
\item[(a)]
Assume that $\nu \ge -1/2$.
If $\sigma > 1/2$, then $B_{\sigma}^{\nu,1/2}$ is of strong type $(p,q)$ for all $p$ and $q$.
If $\sigma = 1/2$, then $B_{\sigma}^{\nu,1/2}$ is of strong type $(p,q) \neq (1,\infty)$.
If $\sigma < 1/2$, then $B_{\sigma}^{\nu,1/2}$ is of strong type $(p,q)$ provided that
$$
\frac{1}q \ge \frac{1}{p} - 2\sigma \qquad \textrm{and} \qquad (p,q) \notin
	\bigg\{ \Big(1,\frac{1}{1-2\sigma}\Big), \Big(\frac{1}{2\sigma},\infty\Big)\bigg\};
$$
moreover, $B_{\sigma}^{\nu,1/2}$ is of weak type $(1,\frac{1}{1-2\sigma})$ and of restricted weak type $(\frac{1}{2\sigma},\infty)$.
\item[(b)]
Assume that $\nu \in (-1,-1/2)$.
If $\sigma > \nu+1$, then $B_{\sigma}^{\nu,1/2}$ is of strong type $(p,q)$ provided that $1/p < \nu+3/2$ and
$1/q > -\nu-1/2$; furthermore, $B_{\sigma}^{\nu,1/2}$ is of weak type $(p,\frac{1}{-\nu-1/2})$ for
$1/p < \nu+3/2$ and of restricted weak type $(\frac{1}{\nu+3/2},q)$ for $1/q \ge -\nu-1/2$.
If $\sigma = \nu+1$, then $B_{\sigma}^{\nu,1/2}$ has the mapping properties as in the case $\sigma > \nu + 1$, except for that
it is not of restricted weak type $(\frac{1}{\nu+3/2}, \frac{1}{-\nu-1/2})$.
If $\sigma < \nu +1$, then $B_{\sigma}^{\nu,1/2}$ is of strong type $(p,q)$ when
$$
\frac{1}p < \nu + \frac{3}2 \qquad \textrm{and} \qquad \frac{1}q > -\nu-\frac{1}2 \qquad \textrm{and} \qquad
	\frac{1}q \ge \frac{1}{p} - 2\sigma;
$$
further, $B_{\sigma}^{\nu,1/2}$ is of weak type $(p,\frac{1}{-\nu-1/2})$ for $1/p < 2\sigma - \nu-1/2$ and of
restricted weak type $(\frac{1}{2\sigma-\nu-1/2},\frac{1}{-\nu-1/2})$ and $(\frac{1}{\nu+3/2},q)$ for
$1/q \ge \nu+3/2 - 2\sigma$.
\end{itemize}
All the results in parts (a) and (b) are sharp in the sense that for no pair $(p,q)$ weak type can be replaced by strong type,
and similarly if restricted weak type $(p,q)$ is claimed, then for no such $(p,q)$ it can be replaced by weak type.
For $(p,q)$ not covered by $(a)$ and $(b)$, $B_{\sigma}^{\nu,1/2}$ is not of restricted weak type $(p,q)$.

All the results stated in (a) remain in force if $B_{\sigma}^{\nu,1/2}$ is replaced
by $A_{\sigma}^{\nu,1/2}$ provided that $\nu > -1/2$, and they are sharp in the sense described above.
\end{theorem}

\begin{proof}
The sharp bounds of Theorem \ref{thm:potk} are the same as in the Jacobi context from Section \ref{ssec:jac} with
parameters $\alpha = \nu$ and $\beta = -1/2$. Thus, the mapping properties stated in the theorem can be concluded from
\cite[Theorem 2.4]{NR3}.
\end{proof}

Mapping properties of the considered potential operators are an important tool in the study of potential and Sobolev spaces associated
with Fourier-Dini expansions. This topic, however, is beyond the scope of this paper.

\section{More on the domains of $\widetilde{\mathbb{L}}^{\nu,H}$ and $\widetilde{\mathbb{J}}^{\a,\b}$} \label{sec:dom}

In this section we shall prove Theorem \ref{thm:dom}. This will be achieved via a series of auxiliary/partial results.
Some of them are slightly more general than needed for our purpose, nevertheless we keep the generality
for the sake of potential independent interest. Throughout, unless stated otherwise, we always consider $\a,\b,\nu > -1$
and $H \in \mathbb{R}$. We shall abbreviate $L^2((0,1),dx)$ to $L^2(0,1)$.

Recall that, by Lemma \ref{lem:tech2}, $\domain\widetilde{\mathbb{L}}^{\nu,1/2} = \domain\widetilde{\mathbb{J}}^{\nu,-1/2}$.
Taking this into account, Theorem \ref{thm:dom} will be proved once we verify items (a)--(d) as follows.
\begin{itemize}
\item[\textbf{(a)}]
If $\b \neq -1/2$, then $\domain\widetilde{\mathbb{L}}^{\nu,H} \neq \domain\widetilde{\mathbb{J}}^{\a,\b}$.
\item[\textbf{(b)}]
If $H \neq 1/2$, then $\domain\widetilde{\mathbb{L}}^{\nu,H} \neq \domain\widetilde{\mathbb{L}}^{\a,1/2}$.
\item[\textbf{(c)}]
If $\a \wedge \nu \le 1$ and $\a \neq \nu$, then
$\domain\widetilde{\mathbb{L}}^{\nu,1/2} \neq \domain\widetilde{\mathbb{L}}^{\a,1/2}$.
\item[\textbf{(d)}]
If $\a,\nu > 1$, then
$\domain\widetilde{\mathbb{L}}^{\nu,1/2} = \domain\widetilde{\mathbb{L}}^{\a,1/2}$.
\end{itemize}

Our line of showing (a)--(c) relies on verifying that suitably chosen functions on $(0,1)$
belong to one of the domains, but do not belong to the other one. On the other hand, proving (d) is less direct.

\subsection{Technical preliminaries}

\begin{lemma} \label{lem:wd}
Let $\a,\b,\nu > -1$ and $H \in \mathbb{R}$. Then
\begin{align*}
\widetilde{\mathbb{L}}^{\nu,H}f & = \mathbb{L}^{\nu}f, \qquad f \in C^2(0,1) \cap \domain\widetilde{\mathbb{L}}^{\nu,H}, \\
\widetilde{\mathbb{J}}^{\a,\b}f & = \mathbb{J}^{\a,\b}f, \qquad f \in C^2(0,1) \cap \domain\widetilde{\mathbb{J}}^{\a,\b}.
\end{align*}
\end{lemma}

\begin{proof}\!\!$^{\spadesuit}$
We only prove the statement regarding $\widetilde{\mathbb{L}}^{\nu,H}$ since treatment of $\widetilde{\mathbb{J}}^{\a,\b}$
is analogous.
\footnote{$\spadesuit$
There is an alternative justification of Lemma \ref{lem:wd} appealing to the theory of Sturm-Liouville operators
and the concepts of minimal/maximal extensions of $\mathbb{L}^{\nu}$ and $\mathbb{J}^{\a,\b}$,
but we prefer to present a simple direct reasoning.}

Let $f \in C^2(0,1) \cap \domain\widetilde{\mathbb{L}}^{\nu,H}$ and $g \in C_c^{\infty}(0,1)$.
The support of $g$ is separated from the endpoints of $(0,1)$, so integrating twice by parts,
see the divergence form of $\mathbb{L}^{\nu}$ in \eqref{id8}, we get
$$
\big\langle \mathbb{L}^{\nu}f,g\big\rangle = \big\langle f, \mathbb{L}^{\nu}g \big\rangle.
$$
Using this and observing that $\mathbb{L}^{\nu}g \in L^2(0,1)$, we can write
\begin{align*}
\mathbb{L}^{\nu}g & = \sum_{n \ge \chi_{\{H+\nu > 0\}}} \big\langle \mathbb{L}^{\nu}g, \psi_n^{\nu,H}\big\rangle \psi_n^{\nu,H}
= \sum_{n \ge \chi_{\{H+\nu > 0\}}} \big\langle g,  \mathbb{L}^{\nu}\psi_n^{\nu,H}\big\rangle \psi_n^{\nu,H} \\
& = \sum_{n \ge \chi_{\{H+\nu > 0\}}} (-1)^{\chi_{\{n=0\}}} \big( \z_n^{\nu,H}\big)^2
	\big\langle g, \psi_n^{\nu,H}\big\rangle \psi_n^{\nu,H} = \widetilde{\mathbb{L}}^{\nu,H}g.
\end{align*}
In particular, $g \in \domain\widetilde{\mathbb{L}}^{\nu,H}$.
Consequently, by the above and self-adjointness of $\widetilde{\mathbb{L}}^{\nu,H}$,
$$
\big\langle \widetilde{\mathbb{L}}^{\nu,H}f, g \big\rangle = \langle f, \widetilde{\mathbb{L}}^{\nu,H}g\big\rangle
= \big\langle f, \mathbb{L}^{\nu}g\big\rangle = \big\langle \mathbb{L}^{\nu}f, g \big\rangle.
$$
Since this holds for all $g \in C_c^{\infty}(0,1)$, which is a dense subset of $L^2(0,1)$, we conclude that
$\widetilde{\mathbb{L}}^{\nu,H}f = \mathbb{L}^{\nu}f$.
\end{proof}

\begin{lemma} \label{lem:rho}
Let $\nu > -1$ and $H \in \mathbb{R}$. Assume that a function $\rho \in C^2(0,1)$ is constant in some neighborhoods of the endpoints
of $(0,1)$. Then
$$
\rho \psi_n^{\nu,H} \in \domain \widetilde{\mathbb{L}}^{\nu,H}, \qquad n \ge \chi_{\{H+\nu > 0\}}.
$$
\end{lemma}

\begin{proof}
Since $\rho$ is bounded on $(0,1)$, it is clear that $\rho \psi_n^{\nu,H} \in L^2(0,1)$.
Moreover, by \eqref{id8} we have
$\mathbb{L}^{\nu}(\rho\psi_n^{\nu,H})=  \rho \mathbb{L}^{\nu}\psi_n^{\nu,H} -\rho'' \psi_n^{\nu,H} - 2\rho'(\psi_n^{\nu,H})'$,
so the assumptions on $\rho$ imply $\mathbb{L}^{\nu}(\rho\psi_n^{\nu,H}) \in L^2(0,1)$.
Integrating twice by parts (using the divergence form of $\mathbb{L}^{\nu}$ in \eqref{id8} and also \eqref{id6},\eqref{id66})
we see that
$$
\big\langle \mathbb{L}^{\nu}\big(\rho\psi_n^{\nu,H}\big), \psi_k^{\nu,H}\big\rangle
= \big\langle \rho\psi_n^{\nu,H}, \mathbb{L}^{\nu}\psi_k^{\nu,H}\big\rangle
= (-1)^{\chi_{\{k=0\}}}\big(\z_k^{\nu,H}\big)^2 \big\langle \rho\psi_n^{\nu,H}, \psi_k^{\nu,H}\big\rangle,
\qquad k \ge \chi_{\{H+\nu > 0\}}.
$$
On the other hand, these are the Fourier-Dini coefficients of $\mathbb{L}^{\nu}(\rho\psi_n^{\nu,H}) \in L^2(0,1)$,
hence they are square-summable in $k$. The conclusion follows.
\end{proof}

\begin{lemma} \label{lem:lagr}
Let $\nu > -1$ and $H \in \mathbb{R}$. Then
$$
f\overline{g'}-f'\overline{g} \big|_0^1 \; := \; \lim_{\epsilon \to 0^+}
\Big[ f(x)\overline{g'(x)}-f'(x)\overline{g(x)}\Big] \bigg|_{x=\epsilon}^{x=1-\epsilon}\; = 0
$$
when $f,g \in C^2(0,1) \cap \domain\widetilde{\mathbb{L}}^{\nu,H}$.
\end{lemma}

\begin{proof}
By self-adjointness of $\widetilde{\mathbb{L}}^{\nu,H}$, we have
$\langle\widetilde{\mathbb{L}}^{\nu,H}f,g\rangle = \langle f,\widetilde{\mathbb{L}}^{\nu,H}g\rangle$.
Thus using Lemma \ref{lem:wd} and then integrating twice by parts$^{\diamondsuit}$ we get
\begin{equation} \label{p777}
0 = \big\langle\widetilde{\mathbb{L}}^{\nu,H}f,g\big\rangle - \big\langle f,\widetilde{\mathbb{L}}^{\nu,H}g\big\rangle
 = \big\langle{\mathbb{L}}^{\nu}f,g\big\rangle - \big\langle f,{\mathbb{L}}^{\nu}g\big\rangle
 = f\overline{g'} - f'\overline{g} \big|_0^1,
\end{equation}
as desired.
\footnote{$\diamondsuit$
Instead of directly integrating by parts twice one could invoke a general theory.
Indeed, the \emph{Lagrange identity} for the Sturm-Liouville operator $\mathbb{L}^{\nu}$ is the last identity
in \eqref{p777} for $f,g \in L^2(0,1)$ such that $f,g,f',g'$ are absolutely continuous in $(0,1)$ and
$\mathbb{L}^{\nu}f, \mathbb{L}^{\nu}g \in L^2(0,1)$; see e.g.\ \cite[Section 2]{weid}.
}
\end{proof}

\begin{lemma} \label{lem:techJ}
Let $\nu,\a > -1$. Assume that $\eta,\xi > 0$, $\eta \neq \xi$, and define $f_{\xi}^{\nu}(x):=\sqrt{x}J_{\nu}(\xi x)$.
Then the Wronskian $W_{\xi,\eta}^{\nu,\a}$ of the system $\{f_{\xi}^{\nu},f_{\eta}^{\a}\}$ expresses as
\begin{align*}
W_{\xi,\eta}^{\nu,\a}(x) & := f_{\xi}^{\nu}(x) \big[ f_{\eta}^{\a}(x)\big]' - \big[f_{\xi}^{\nu}(x)]' f_{\eta}^{\a}(x) \\
& = (\a-\nu) J_{\nu}(\xi x) J_{\a}(\eta x) + x \big[ \xi J_{\nu+1}(\xi x) J_{\a}(\eta x) - \eta J_{\nu}(\xi x) J_{\a+1}(\eta x)\big].
\end{align*}
\end{lemma}

\begin{proof}
This is a direct computation based on the formula, cf.\ \eqref{id3}, $J'_{\nu}(x) = \frac{\nu}x J_{\nu}(x) - J_{\nu+1}(x)$.
We leave the details to the reader.
\end{proof}

In what follows we will use cutoff functions $\rho_0$ and $\rho_1$ defined as follows.
Let $0 \le \rho_0 \le 1$ be a smooth function on $(0,1)$ such that $\rho_0 \equiv 1$ on $(0,1/4)$ and $\rho_0 \equiv 0$ on $(3/4,1)$.
Then put $\rho_1 := 1 - \rho_0$.

\subsection{Proof of (a)}
We assume that $\b \neq -1/2$ and will show that $\domain\widetilde{\mathbb{L}}^{\nu,H} \neq \domain\widetilde{\mathbb{J}}^{\a,\b}$.
It is convenient to consider two cases.

\noindent \textbf{Case 1:} $\beta \neq 1/2$.
Clearly, $\psi_1^{\nu,H} \in C^2(0,1) \cap\domain\widetilde{\mathbb{L}}^{\nu,H}$.
In view of Lemma \ref{lem:wd}, it is enough to check that $\mathbb{J}^{\a,\b}\psi_1^{\nu,H} \notin L^2(0,1)$.
We have
$\mathbb{J}^{\a,\b}\psi_1^{\nu,H} = (\mathbb{J}^{\a,\b} - \mathbb{L}^{\nu}) \psi_1^{\nu,H} + \mathbb{L}^{\nu}\psi_1^{\nu,H}$.
Here $\mathbb{L}^{\nu}\psi_1^{\nu,H} \in L^2(0,1)$, but $(\mathbb{J}^{\a,\b} - \mathbb{L}^{\nu}) \psi_1^{\nu,H} \notin L^2(0,1)$,
see \eqref{LJdiff}. Consequently, $\mathbb{J}^{\a,\b}\psi_1^{\nu,H} \notin L^2(0,1)$.

\noindent \textbf{Case 2:} $\beta = 1/2$.
It is known, cf.\ \cite[Theorem 3.1]{NR2}, that $\domain\widetilde{\mathbb{J}}^{\a,1/2} = \domain\widetilde{\mathbb{L}}^{\a}$,
where $\widetilde{\mathbb{L}}^{\a}$ is the self-adjoint extension of $\mathbb{L}^{\a}$ associated with Fourier-Bessel expansions.
Thus, it suffices to show that $\domain\widetilde{\mathbb{L}}^{\nu,H} \neq \domain\widetilde{\mathbb{L}}^{\a}$.

Consider $f = \rho_1 \psi_1^{\nu,H}$ and $g=\psi_1^{\a}$, where $\psi_1^{\a}(x)$ is the first eigenfunction of
$\widetilde{\mathbb{L}}^{\a}$ (up to a positive constant factor, it is equal to $\sqrt{x}J_{\a}(\z_1^{\a}x)$, where $\z_1^{\a}$ is
the first positive zero of $J_{\a}$). By Lemma \ref{lem:rho}, $f \in C^2(0,1) \cap \domain\widetilde{\mathbb{L}}^{\nu,H}$,
and obviously $g \in C^2(0,1) \cap \domain\widetilde{\mathbb{L}}^{\a}$.

Aiming at a contradiction, assume that $\domain\widetilde{\mathbb{L}}^{\nu,H} = \domain\widetilde{\mathbb{L}}^{\a}$.
In Lemma \ref{lem:techJ} we let $\xi=\z_1^{\nu,H}$ and $\eta = \z_1^{\a}$ and observe that
$fg'-f'g$ is equal, up to a positive constant factor $c$, $W_{\xi,\eta}^{\nu,\a}$ on $(3/4,1)$, and
$fg'-f'g$ vanishes on $(0,1/4)$. Thus, by Lemma \ref{lem:techJ},
$$
fg'-f'g\big|_0^1 = c\, W_{\xi,\eta}^{\nu,\a}(1) = -c \z_1^{\a} J_{\nu}\big(\z_1^{\nu,H}\big) J_{\a+1}\big( \z_1^{\a}\big) \neq 0,
$$
where the last relation holds since positive zeros of $J_{\a}$ interlace those of $J_{\a+1}$,
and those of $J_{\a,H}$ (cf.\ Section \ref{sec:Bes}).
In view of Lemma~\ref{lem:lagr}, this is a contradiction which proves that
$\domain\widetilde{\mathbb{L}}^{\nu,H} \neq \domain\widetilde{\mathbb{L}}^{\a}$.

\subsection{Proof of (b)}
We assume that $H \neq 1/2$ and will show that $\domain\widetilde{\mathbb{L}}^{\nu,H} \neq \domain\widetilde{\mathbb{L}}^{\a,1/2}$.
The argument is similar to that from Case 2 in the proof of (a).

Consider $f = \rho_1 \psi_1^{\nu,H}$ and $g=\psi_1^{\a,1/2}$. Then $f \in C^2(0,1) \cap \domain\widetilde{\mathbb{L}}^{\nu,H}$
(cf.\ Lemma \ref{lem:rho}) and $g \in C^2(0,1) \cap \domain\widetilde{\mathbb{L}}^{\a,1/2}$.
In Lemma \ref{lem:techJ} we let $\xi = \z_1^{\nu,H}$ and $\eta = \z_1^{\a,1/2}$, and get
$$
fg'-f'g\big|_0^1 = c\, W_{\xi,\eta}^{\nu,\a}(1) = c\, (H-1/2) J_{\nu}\big( \z_1^{\nu,H}\big) J_{\a}\big(\z_1^{\a,1/2}\big) \neq 0,
$$
where $c>0$ and the second equality follows by combining Lemma \ref{lem:techJ} with \eqref{id6}.
Now the conclusion follows by Lemma \ref{lem:lagr}.

\subsection{Proof of (c)}
We assume that $\a \wedge \nu \le 1$ and $\nu \neq \a$ and will show that
$\domain\widetilde{\mathbb{L}}^{\nu,1/2} \neq \domain\widetilde{\mathbb{L}}^{\a,1/2}$. For symmetry reasons,
we may assume that $\nu < \a$, thus $\nu \in (-1,1]$.
Then it is convenient to consider two cases.

\noindent \textbf{Case 1:} $\a \neq -\nu$.
Aiming at a contradiction, we assume that
$\psi_1^{\nu,1/2} \in \domain\widetilde{\mathbb{L}}^{\nu,1/2} \cap \domain\widetilde{\mathbb{L}}^{\a,1/2}$.
Then, using Lemma \ref{lem:wd}, we have
$$
L^2(0,1) \ni \widetilde{\mathbb{L}}^{\nu,1/2}\psi_1^{\nu,1/2} - \widetilde{\mathbb{L}}^{\a,1/2}\psi_1^{\nu,1/2}
 = \mathbb{L}^{\nu}\psi_1^{\nu,1/2} - \mathbb{L}^{\a}\psi_1^{\nu,1/2}
 = \frac{\nu^2-\a^2}{x^2} \psi_1^{\nu,1/2}.
$$
However, in view of \eqref{as1} the function in question does not belong to $L^2(0,1)$ for $\nu \le 1$,
due to its order of magnitude at $0^+$. This contradiction shows that
$\psi_1^{\nu,1/2} \in \domain\widetilde{\mathbb{L}}^{\nu,1/2} \setminus \domain\widetilde{\mathbb{L}}^{\a,1/2}$.

\noindent \textbf{Case 2:} $\a = -\nu$. Thus $\nu \in (-1,1)\setminus\{0\}$.
We need to prove that $\domain\widetilde{\mathbb{L}}^{\nu,1/2} \neq \domain\widetilde{\mathbb{L}}^{-\nu,1/2}$.
The argument that follows is similar to those from the proofs of Case 2 in (a) and (b).

Consider $f=\rho_0 \psi_1^{\nu,1/2}$ and $g=\psi_1^{-\nu,1/2}$. Then
$f \in C^2(0,1) \cap \domain\widetilde{\mathbb{L}}^{\nu,1/2}$ (see Lemma \ref{lem:rho}) and
$g \in C^2(0,1) \cap \domain\widetilde{\mathbb{L}}^{-\nu,1/2}$. In Lemma \ref{lem:techJ} we let
$\xi = \z_1^{\nu,1/2}$ and $\eta = \z_1^{-\nu,1/2}$ and obtain
$$
fg'-f'g\big|_0^1 = c \lim_{x\to 0^+} W_{\xi,\eta}^{\nu,\a}(x) =
\frac{-2 c\,  \nu}{\Gamma(1-\nu)\Gamma(1+\nu)}\Big(\frac{\xi}{\nu}\Big)^{\nu}
= -\frac{2c}{\pi} \sin(\pi \nu) \bigg(\frac{\z_1^{\nu,1/2}}{\z_1^{-\nu,1/2}}\bigg)^{\nu} \neq 0;
$$
here $c>0$ and the second equality follows by combining Lemma \ref{lem:techJ} with \eqref{as1}.
To get the third equality we used a well-known identity for the gamma function.
The desired conclusion now follows by Lemma~\ref{lem:lagr}.

\subsection{Proof of (d)}
We assume that $\nu,\a > 1$ and must verify that $\domain\widetilde{\mathbb{L}}^{\nu,1/2} = \domain\widetilde{\mathbb{L}}^{\a,1/2}$.
This can be done by repeating the reasoning from the proof Lemma \ref{lem:tech2}, where also an analogue of Lemma~\ref{lem:tech}
is necessary. The only additional result needed is that $f/x^2 \in L^2(0,1)$ whenever $f \in \domain\widetilde{\mathbb{L}}^{\nu,1/2}$,
since in the present scenario the counterpart of the function $F^{\nu}$ is in absolute value comparable to $x^{-2}$.
We shall leave the straightforward details to interested readers, except for the aforementioned auxiliary result which follows
from the lemma below. The lemma can be regarded as a variant of the classic Rellich inequality \cite{Rel}
in the framework of $\widetilde{\mathbb{L}}^{\nu,1/2}$, which might be of independent interest.

\begin{lemma} \label{lem:x2}
Let $\nu > 1$. Then
\begin{equation} \label{Rell}
\big\| f/x^2 \big\|_{L^2(0,1)} \le \frac{1}{\nu^2-1} \big\| \widetilde{\mathbb{L}}^{\nu,1/2}f \big\|_{L^2(0,1)},
\qquad f \in \domain \widetilde{\mathbb{L}}^{\nu,1/2}.
\end{equation}
\end{lemma}

\begin{proof}
We first show the inequality in \eqref{Rell} assuming that $f \in \spann\{\psi_n^{\nu,1/2}: n\ge 1\}$.
The method we use is an adaptation of a reasoning that appears to be known, at least as a folklore, and in the version
presented below was communicated to us by G.\ Metafune \cite{Met}.
Observe that since $\nu > 1$ and $f$ is a finite linear combination of the $\psi_n^{\nu,1/2}$, it follows by \eqref{as1}
that $f/x^2 \in L^2(0,1)$. We also have $\widetilde{\mathbb{L}}^{\nu,1/2}f = \mathbb{L}^{\nu}f$.

Set $\zeta:= \nu^2-1/4 > 3/4$. Integrating by parts, we get
\begin{align*}
\big\langle \mathbb{L}^{\nu}f, f/x^2 \big\rangle & =
-\int_0^1 \frac{f''(x)\overline{f(x)}}{x^2} \, dx + \zeta \int_0^1 \frac{|f(x)|^2}{x^4}\, dx \\
& = - \frac{f'(x)\overline{f(x)}}{x^2}\bigg|_0^1 + \int_0^1 \frac{|f'(x)|^2}{x^2}\, dx - 2\int_0^1 \frac{f'(x)\overline{f(x)}}{x^3}\, dx
	+ \zeta\int_0^1 \frac{|f(x)|^2}{x^4}\, dx.
\end{align*}
Here the integrated term vanishes, since $\psi_n^{\nu,1/2}(x) = \mathcal{O}(x^{\nu+1/2})$,
$\frac{d}{dx}\psi_n^{\nu,1/2}(x) = \mathcal{O}(x^{\nu-1/2})$, $x \to 0^+$, and $\frac{d}{dx}\psi_n^{\nu,1/2}(x)\big|_{x=1} =0$;
all this can easily be verified with the aid of \eqref{as1}, \eqref{id3} and \eqref{id6}.
Therefore, we obtain
\begin{equation} \label{eqd}
\big\langle \mathbb{L}^{\nu}f,f/x^2\big\rangle = \big\| f'/x\big\|^2_{L^2(0,1)}
	- 2\big\langle f'/x, f/x^2\big\rangle + \zeta \big\|f/x^2\big\|^2_{L^2(0,1)}.
\end{equation}
On one hand, by the Cauchy-Schwarz inequality we have
\begin{equation} \label{equ}
\big| \big\langle \mathbb{L}^{\nu}f,f/x^2 \big\rangle\big|
	\le \big\| \mathbb{L}^{\nu}f \big\|_{L^2(0,1)} \big\|f/x^2\big\|_{L^2(0,1)}.
\end{equation}
On the other hand, we now bound the right-hand side of \eqref{eqd} from below.

We begin by noting that, again by the Cauchy-Schwarz inequality,
\begin{equation} \label{eqCS}
\big| \big\langle  f'/x, f/x^2 \big\rangle\big| \le \big\|f'/x\big\|_{L^2(0,1)} \big\|f/x^2\big\|_{L^2(0,1)}.
\end{equation}
Next, we claim the Hardy-type inequality
\begin{equation} \label{hardy}
\big\|f/x^2\big\|_{L^2(0,1)} \le \frac{2}3 \big\| f'/x\big\|_{L^2(0,1)}.
\end{equation}
To see this, write $f(x)/x = \int_0^1 f'(sx)\,ds$, take absolute squares of both sides and integrate with respect
to $x^{-2}dx$ over $(0,1)$. The result will be
$$
\big\| f/x^2 \big\|_{L^2(0,1)}^2 = \int_0^1 \bigg| \int_0^1 f'(sx)\, ds\bigg|^2 \, \frac{dx}{x^2}.
$$
Then apply Minkowski's integral inequality to obtain
$$
\big\|f/x^2\big\|_{L^2(0,1)} \le \int_0^1 \bigg( \int_0^1 \big|f'(sx)\big|^2 \, \frac{dx}{x^2} \bigg)^{1/2}\, ds
 \le \int_0^1 \big\| f'/x \big\|_{L^2(0,1)} \sqrt{s}\, ds = \frac{2}3 \big\|f'/x\big\|_{L^2(0,1)}
$$
and the claim follows.

Now, let $Y := \|f'/x\|_{L^2(0,1)}/ \|f/x^2\|_{L^2(0,1)}$. By \eqref{hardy} we have $Y \ge 3/2$ and we can use \eqref{eqCS} to write
\begin{align*}
& \big\|f'/x\big\|_{L^2(0,1)}^2 - 2\big\langle  f'/x, f/x^2 \big\rangle + \zeta \big\|f/x^2\big\|_{L^2(0,1)}^2 \\
& \qquad \ge \big\|f'/x\big\|_{L^2(0,1)}^2 - 2\big\|f'/x\big\|_{L^2(0,1)}\big\|f/x^2\big\|_{L^2(0,1)}
	+ \zeta\big\|f/x^2\big\|_{L^2(0,1)}^2 \\
& \qquad = \big[ (Y-3/2)(Y-1/2) + \zeta -3/4 \big] \, \big\| f/x^2\big\|_{L^2(0,1)}^2 \ge (\zeta-3/4)\big\|f/x^2\big\|_{L^2(0,1)}^2.
\end{align*}
Combining the above with \eqref{eqd} and the upper bound \eqref{equ}, we conclude the estimate in \eqref{Rell}
for $f \in \spann\{\psi_n^{\nu,1/2} : n \ge 1\}$.

We now pass to the general case when $f \in \domain\widetilde{\mathbb{L}}^{\nu,1/2}$. Let
$$
f_N = \sum_{n=1}^N \big\langle f, \psi_n^{\nu,1/2} \big\rangle \psi_n^{\nu,1/2}, \qquad N \ge 1,
$$
be the partial sums of the Fourier-Dini expansion of $f$; clearly, $f_N \in \spann\{\psi_n^{\nu,1/2} : n \ge 1\}$.
Then $f_N \to f$ and $\widetilde{\mathbb{L}}^{\nu,1/2}f_N \to \widetilde{\mathbb{L}}^{\nu,1/2}f$ in $L^2(0,1)$, as $N \to \infty$.
Moreover, we can choose a subsequence $\{f_{N_k}\}$ of $\{f_N\}$ for which the convergence is pointwise a.e.
Using Fatou's lemma together with the already proved bound \eqref{Rell} for $f_{N_k}$ we infer that
\begin{align*}
\big\|f/x^2\big\|_{L^2(0,1)}^2 & = \int_0^1 \liminf_{k} \frac{|f_{N_k}(x)|^2}{x^4}\, dx
\le \liminf_{k} \big\| f_{N_k}/x^2\big\|_{L^2(0,1)}^2 \\
& \le \frac{1}{(\nu^2-1)^2} \liminf_{k} \big\|\widetilde{\mathbb{L}}^{\nu,1/2}f_{N_k}\big\|_{L^2(0,1)}^2
= \frac{1}{(\nu^2-1)^2} \big\| \widetilde{\mathbb{L}}^{\nu,1/2}f\big\|_{L^2(0,1)}^2.
\end{align*}
This implies \eqref{Rell} and finishes the proof.
\end{proof}

\begin{remark}
The multiplicative constant in \eqref{Rell} is optimal, see \cite[Proposition 3.10]{MSS}.
\end{remark}

\section*{Acknowledgments}
We are grateful to Giorgio Metafune for showing us the method used in the proof of Lemma \ref{lem:x2}.
We also thank the referees for a constructive criticism that enriched the presentation and led us to
certain corrections.


\end{document}